\documentclass[a4paper,12pt,twoside]{amsart}
\usepackage{amsmath,amsthm,amsfonts}

\newtheorem{thm}{Theorem}[section]

\newtheorem{cor}[thm]{Corollary}
\newtheorem{lem}[thm]{Lemma}

\newtheorem{prop}[thm]{Proposition}

\newtheorem{definition}{Definition}
\newtheorem{defn}[thm]{Definition}
\newtheorem{defi}[thm]{Definition}

\newtheorem{nota}[thm]{Notation}

\theoremstyle{definition}

\newtheorem{rem}[thm]{Remark}

\newcommand{\R}{{\mathbb{R}}}

\newcommand{\E}{{\mathbb{E}}}
\newcommand{\Q}{{\mathbb{Q}}}
\newcommand{\C}{{\mathbb{C}}}

\newcommand{\Z}{{\mathbb{Z}}}
\newcommand{\N}{{\mathbb{N}}}
\newcommand{\T}{{\mathbb{T}}}
\newcommand{\PP}{{\mathbb{P}}}

\newcommand{\Cal}{\mathcal}
\newcommand{\cH}{\mathcal{H}}

\def \diag {{\rm diag} \,}

\def\e{{\rm e}}

\def\tn {{\underline n}}  \def\tp {{\underline p}}
\def\j{{\underline j}} \def\k{{\underline k}}
 \def\t{{ \mathbf t}}
 \def\el{{\underline \ell}}
\def\t{{\underline t}} \def\n {{\underline n}}
\def\0{{\underline 0}} \def\m {{\underline m}}
\def\a {{\underline a}} 
\def\u{{\underline u}}

\def\stm0{{\setminus \{\0\}}}

\def\eop{\qed}

\begin{document}
\baselineskip 15pt
\parindent=0mm

\title[CLT for commutative
semigroups of toral endomorphisms] {Central limit theorem for
commutative \\ semigroups of toral endomorphisms}

\bigskip
\date{}
\author{Guy Cohen and Jean-Pierre Conze}
\address{Guy Cohen, \hfill \break Dept. of Electrical Engineering,
\hfill \break Ben-Gurion University, Israel}
\email{guycohen@ee.bgu.ac.il}
\address{Jean-Pierre Conze,
\hfill \break IRMAR, CNRS UMR 6625, \hfill \break University of
Rennes I, Campus de Beaulieu, 35042 Rennes Cedex, France}
\email{conze@univ-rennes1.fr}

\subjclass[2010]{Primary: 60F05, 28D05, 22D40; Secondary: 47A35,
47B15} \keywords{Central Limit Theorem, $\Z^d$-action, semigroup of
endomorphisms, toral automorphisms, powers of barycenters, rotated
process, moments and mixing, $S$-units}

\begin{abstract} Let $\Cal S$ be an abelian finitely generated semigroup
of endomorphisms of a probability space $(\Omega, {\Cal A}, \mu)$,
with $(T_1, ..., T_d)$ a system of generators in ${\Cal S}$. Given
an increasing sequence of domains $(D_n) \subset \N^d$, a question
is the convergence in distribution of the normalized sequence
$|D_n|^{-\frac12} \sum_{{\k} \, \in D_n} \, f \circ T^{\,{\k}}$, or
normalized sequences of iterates of barycenters $Pf = \sum_j p_j f
\circ T_j$, where $T^{\k}= T_1^{k_1} ... T_d^{k_d}$, ${\k}= (k_1,
..., k_d) \in {\N}^d$.

After a preliminary spectral study when the action of $\Cal S$ has a
Lebesgue spectrum, we consider totally ergodic $d$-dimensional
actions given by commuting endomorphisms on a compact abelian
connected group $G$ and we show a CLT, when $f$ is regular on $G$.
When $G$ is the torus, a criterion of non-degeneracy of the variance
is given.
\end{abstract}

\maketitle

\tableofcontents

\section*{\bf Introduction}

Let $\Cal S$ be an abelian finitely generated semigroup of
endomorphisms of a probability space $(\Omega, {\Cal A}, \mu)$. Each
$T \in \Cal S$ is a measurable map from $\Omega$ to $\Omega$
preserving the probability measure $\mu$. For $f \in L^1(\mu)$ a
random field is defined by $(f(T.)_{T \in {\Cal S}})$ for which
limit theorems can be investigated: law of large numbers, behavior
in distribution.

By choosing a system $(T_1, ..., T_d)$ of generators in ${\Cal S}$,
every $T \in {\Cal S}$ can be represented\footnote{We underline the
elements of $\N^d$ or $\Z^d$ to distinguish them from the scalars
and write $T^\k f$ for $f \circ T^\k$.} as $T = T^{\k}= T_1^{k_1}
... T_d^{k_d}$, for ${\k}= (k_1, ..., k_d) \in {\N}^d$. Given an
increasing sequence of domains $(D_n) \subset \N^d$, a question is
the asymptotic normality of the standard normalized sequence and the
``multidimensional periodogram" respectively defined by
\begin{eqnarray}
|D_n|^{-\frac12} \sum_{{\k} \, \in D_n} \, T^{\,{\k}} f, \ \
|D_n|^{-\frac12} \sum_{\k \, \in D_n} \, e^{2\pi i \langle \k,
\theta \rangle} \, T^{{\k}} f, \ f \in L^2_0(\mu), \ \theta \in
\R^d. \label{ergSumDn}
\end{eqnarray}

Let us take for $(\Omega, {\Cal A}, \mu)$ a compact abelian group
$G$ endowed with its Borel $\sigma$-algebra ${\Cal A}$ and its Haar
measure $\mu$. In this framework, the first examples of dynamical
systems satisfying a CLT in a class of regular functions are due to
R.~Fortet and M.~Kac for endomorphisms of $\T^1$. In 1960 V. Leonov
(\cite{Leo60b}) showed that, if $T$ is an ergodic endomorphism of
$G$, then the CLT is satisfied for regular functions $f$ on $G$.

The $d$-dimensional extension of this situation leads to the
question of  validity of a CLT for algebraic actions on an abelian
compact group $G$, i.e., when $T^{\k}$ in Formula (\ref{ergSumDn})
is given by an action of $\N^d$ on $G$ by automorphisms or more
generally endomorphisms.

By composition, one obtains an action by isometries on ${\Cal H} =
L^2_0(\mu)$, the space of square integrable functions $f$ such that
$\mu(f) =0$. The spectral analysis of this action is the content of
Section \ref{specAnaly} where the methods of summation are also
discussed.

In Section \ref{endoGroup} we consider $d$-dimensional actions given
by commuting endomorphisms on a connected abelian compact group $G$.
For a regular function $f$ on $G$, a CLT is shown for the above
normalized sequence (Theorem \ref{tclPol0A}) and other summation
methods like barycenters (Theorem \ref{tclBaryc}), as well as a
criterion of non-degeneracy of the variance when $G$ is a torus. The
barycenters yield a class of operators with a polynomial decay to
zero of the iterates applied to regular functions. This contrasts
with the spectral gap property for non amenable group actions by
automorphisms on tori.

When $(D_n)$ is a sequence of $d$-dimensional cubes, for the
periodogram in (\ref{ergSumDn}), given a function $f$ in $L_0^2(G)$,
the CLT is obtained for almost every $\theta$, without regularity
requirement.

When $G$ is a torus, using the exponential decay of correlation, the
CLT can be shown for a class of functions with weak regularity and
one can characterize the case of degeneracy in the limit theorem. In
an appendix, classical results on the construction of $\Z^d$-actions
by automorphisms are recalled.

One of our aims was to extend to a larger class of semigroups of
actions by endomorphisms the CLT proved by T. Fukuyama and B. Petit
(\cite{FuPe01}) for semigroups generated by coprime integers on the
circle. Their result corresponds, in our framework, to sums taken on
an increasing sequence of triangles in $\N^2$.

After completion of a first version of this paper, we were informed
by B. Weiss of the recent paper by M. Levine (\cite{Lev13}) in which
the CLT and a functional version of it are obtained for actions by
endomorphisms on the torus. The proof of the CLT for sums on
$d$-dimensional ``rectangles" is based in both approaches, as well
as in \cite{FuPe01}, on results on $S$-units. In the present paper
we use the formalism of cumulants and the result of Schmidt and Ward
(\cite{SchWar93}) on mixing of all orders for connected groups
deduced from a deep result on $S$-units. We make use of the spectral
measure which is well adapted to a ``quenched" CLT and a CLT along
different types of summation sequences, in particular the iterates
of barycenters. The connectedness of the group $G$ is assumed only
for the CLT.

\vskip 3mm
\section{\bf Spectral analysis} \label{specAnaly}

In this section we consider the general framework of the action of
an abelian finitely generated semigroup ${\Cal S}$ of isometries on
a Hilbert space ${\Cal H}$. We have in mind the example of a
semigroup $\Cal S$ of endomorphisms of a compact abelian group $G$
acting on ${\Cal H} = L_0^2(G, \mu)$, with $\mu$ the Haar measure of
$G$.

With the notations of the introduction, every $T \in {\Cal S}$ is
represented as $T = T^{\el} = T_1^{\ell_1} ... T_d^{\ell_d}$, where
$(T_1, ..., T_d)$ is a system of generators in ${\Cal S}$ and
${\el}= (\ell_1, ..., \ell_d) \in \N^d$.

Given $f \in \Cal H$, for $d > 1$, there are various choices of the
sets of summation $D_n$ for the field $(T^{\el} f, \el \in \N^d)$.
We discuss this point, as well as the behavior of the associated (by
discrete Fourier transform) kernels. The second subsection is
devoted to the spectral analysis of the $d$-dimensional action.

\subsection{\bf Summation and kernels, barycenters}
\label{LebSpec}

\

If $(D_{n})_{n \geq 1}$ is  a sequence of subsets of $\N^d$, the
corresponding rotated sum and kernel are respectively: $\sum_{{\el}
\in D_n} \e^{2\pi i\langle {\el },\, \theta \rangle}T^\el f$ and
$\frac1{|D_n|}\big|\sum_{{\el} \in D_n} \e^{2\pi i\langle {\el },\,
{t}\rangle}\big|^2$. The simplest choice for $(D_n)$ is an
increasing family of $d$-dimensional squares or rectangles.

\begin{nota} \label{sumSeq} {\rm
More generally, we will call {\it summation sequence} a uniformly bounded sequence
$(R_n)$ of functions from $\N^d$ to $\R^+$. It could be also defined
on $\Z^d$, but for simplicity in this section we consider summation
for $\el \in \N^d$. If $T = (T^\el)_{\el \in \N^d}$ is a semigroup of isometries, an
associated sequence of operators on $\Cal H$ can be defined by
$$R_n(T) : f \in {\Cal H} \to R_n(T)f:= \sum_{\el \in \N^d} R_n(\el) T^\el f.$$

We will write simply $R_n$ instead of $R_n(T)$. By introducing a
rotation term, these operators extend to a family of operators
$R_n^\theta$, for $\theta \in \R^d$,
$$f \to R_n^\theta f:= \sum_{\el \in \N^d} R_n(\el) \, e^{2\pi i
\langle \el, \theta \rangle} \, T^\el f.$$ We have $\|\sum_{\el \in
\N^d} R_n(\el) e^{2\pi i \langle \el, . \rangle}\|_{L^2(\T^d,dt)}^2
= \sum_{\el \in \N^d} |R_n(\el)|^2$. Taking the discrete Fourier
transform, we associate to $R_n$ the normalized ``kernel" $\tilde
R_n$ defined on $\T^d$ by:
$$\tilde R_n(t) = {|\sum_{\el \in \N^d} R_n(\el)  e^{2\pi i \langle \el,
t \rangle}|^2 \over \sum_{\el \in \N^d} |R_n(\el)|^2}.$$  }
\end{nota}

\begin{definition} {\rm
We say that $(R_n)$ is {\it regular} if $(\tilde R_n)_{n \geq 1}$
weakly converges to a measure $\zeta$ on $\T^d$, i.e., $\int_{\T^d}
\tilde R_n \,\varphi \, dt \underset{n \to \infty}
{\longrightarrow}\int_{\T^d} \varphi \, d\zeta$ for every continuous
function $\varphi$ on $\T^d$. If $(R_n)$ is regular and $\zeta$ is
the Dirac mass at 0, we say that $(R_n)$ is a {\it F\o{}lner
sequence}.}
\end{definition}
If  $(R_n) = (1_{D_n})$ is associated to a sequence of sets $D_n
\subset \N^d$, one easily proves that $(R_n)$ is a F\o{}lner
sequence if and only if $(D_n)$ satisfies the F\o{}lner condition:
\begin{eqnarray}
\lim_{n\to\infty} |D_n|^{-1} |(D_n+{\tp}) \, \cap \, D_n| =1, \
\forall {\tp}\in \Z^d. \label{FolnCond}
\end{eqnarray}

\vskip 2mm {\bf Examples.} \ a) {\it Squares and rectangles.}  \
Using the usual one-dimensional Fej\'er kernel $K_N(t) = {1\over N}
({\sin \pi N t \over \sin \pi t })^2$, the $d$-dimensional Fej\'er
kernels on $\T^d$ corresponding to rectangles are defined by
$K_{N_1,...,N_d}(t_1,..., t_d) = K_{N_1}(t_1)\cdots K_{N_d}(t_d), \
{\underline N} =(N_1, ..., N_d) \in\mathbb N^d$. They are the
kernels associated to $D_{\underline N} := \{{ \k}\in\mathbb N^d:\ {
k}_i\le N_i, 1 \le i\le d \}.$

b) A family of examples satisfying (\ref{FolnCond}) can be obtained as
follows: take a non-empty domain $D \subset \R^d$ with {\it smooth}
boundary and finite area and put $D_n=\lambda_n D\cap \Z^d$, where
$(\lambda_n)$ is an increasing sequence of real numbers tending to
$+\infty$.

c) {\it Kernels with unbounded gaps} \ If $(D_\n)$ is a (non
F\o{}lner) sequence of domains such that $\lim_\n
\frac{|(D_\n+{\tp}) \cap D_\n|} {|D_\n|}=0$ for ${\tp}\not={{0}}$,
then $\lim_\n(\tilde R_\n*\varphi)(\theta)=\int_{\mathbb T^d}
\varphi(t)dt$, for every ${\theta}\in\mathbb T^d$ and $\varphi$
continuous, where $(\tilde R_\n)$ is the kernel associated to
$(D_n)$.

For example, let $k_j$ be a sequence with
$k_{j+1}-k_j\to\infty$ and put $D_n=\{k_j:\ 0\le j\le n-1\}$. For
$p\not=0$ the number of solutions of $k_j-k_\ell=p$, for
$j,\ell\ge0$ is finite, so that $\lim_{n\to\infty}\frac{|(D_n+{
p})\, \cap \, D_n|}{|D_n|}=0$ for $p\not=0$.

d) {\it Iteration of barycenter operators} \ Let $T_1, ..., T_d$ be $d$ commuting
unitary operators on a Hilbert space ${\Cal H}$. If $(p_1, ..., p_d)$ is a probability
vector such that $p_j > 0, \forall j$, for $\theta = (\theta_1,\theta_2,
..., \theta_d) \in \T^d$, we will consider the barycenter operators
defined on $\Cal H$ by
\begin{eqnarray}
P: f \to \sum_{j=1}^d  \, p_j \, T_j f, \  \ P_\theta: f \to
\sum_{j=1}^d \, p_j \, e^{2\pi i \theta_j} \, T_j f. \label{defBary}
\end{eqnarray}
The iteration of $P$ or $P_\theta$ gives a method of summation which
is not of F\o{}lner type.

\subsection{\bf Lebesgue spectrum, variance} \label{orthogIncr}

\

Let $\Cal S$ be a finitely generated {\it torsion free} commutative
group of unitary operators on a Hilbert space $\Cal H$. Let $(T_1,
..., T_d)$ be a system of independent generators in ${\Cal S}$. Each
element of $\Cal S$ can be written in a unique way as $T^{\el}=
T_1^{\ell_1} ... T_d^{\ell_d}$, with ${\el}= (\ell_1, ..., \ell_d)
\in \Z^d$, and $\el \to T^\el$ defines a unitary representation of
$\Z^d$ in $\Cal H$.

 For every $f \in \cH$, there is a positive finite measure $\nu_f$ on $\T^d$
such that, for every $\el \in \Z^d$, $\hat \nu_f({\el}) = \langle
T^{\ell_1}... T^{\ell_d} f, f \rangle$.
\begin{defi} \label{mixLeb} {\rm Recall that the action of $\Cal S$ on ${\Cal H}$
has a {\it Lebesgue spectrum}, if there exists ${\Cal K}_0$, a
closed subspace of ${\Cal H}$, such that the subspaces $T^\el {\Cal
K}_0$ are pairwise orthogonal and span a dense subspace in  ${\Cal
H}$.}\end{defi} The Lebesgue spectrum property implies mixing, i.e.,
$\lim_{\|\n\| \to \infty}|\langle T^\n f, g\rangle| = 0, \ \forall
f, g \in \Cal H$. With the Lebesgue spectrum property, for every $f
\in \Cal H$, the corresponding spectral measure $\nu_f$ of $f$ on
$\T^d$ has a density $\varphi_f$. A change of basis induces for the
spectral density the composition by an automorphism acting on
$\T^d$.

A family of examples of $\Z^d$-actions by unitary operators is
provided by the action of a group of commuting automorphisms on a
compact abelian group $G$. In the present paper, we will focus
mainly on this class of examples.

\begin{nota} \label{gammaj2} {\rm For any orthonormal basis $(\psi_j)_{j \in J}$
of ${\Cal K}_0$, the family $(T^\el \psi_j)_{j \in J, \, \el \in
\Z^d}$ is an orthonormal basis of ${\Cal H}$. Let $\cH_j$ be the
closed subspace (invariant by the $\Z^d$-action) generated by $(T^\n
\psi_j)_{\n \in \Z^d}$.

We set $a_{j, \, \tn}:= \langle f, T^{\tn} \psi_j \rangle$, $j \,
\in J$. Let $f_j$ be the orthogonal projection of $f$ on $\cH_j$ and
$\gamma_j$ an everywhere finite square integrable function on $\T^d$
with Fourier coefficients $a_{j, \, \n}$.

The spectral measure of $f$ is the sum of the spectral measures of
$f_j$. For $f_j$, the density of the spectral measure is
$|\gamma_j|^2$. Therefore, by orthogonality of the subspaces
$\cH_j$, the density of the spectral measure of $f$ is $\varphi_f(t)
= \sum_{j \in J} |\gamma_j(t)|^2$.}\end{nota}

We have: $\int_{\T^d} \sum_{j \in J} |\gamma_j(\theta)|^2 \ d\theta
= \sum_{j \, \in J} \sum_{\tn\in\mathbb Z^d} |a_{j, \tn}|^2 =
\int_{\T^d} \varphi_f(\theta) \ d\theta =\|f\|^2 < \infty$ and the
set $\Lambda_0(f):= \{\theta \in \T^d: \sum_{j \, \in J}
|\gamma_j(\theta)|^2 < \infty\}$ has full measure.

For $\theta$ in $\T^d$, let $M_\theta f$ in ${\Cal K}_0$ (with orthogonal ``increments")
be defined by:
\begin{eqnarray}
M_\theta f := \sum_j \gamma_j(\theta) \, \psi_j. \label{Mthetaf}
\end{eqnarray}
Under the condition $ \sum_{j\in J}\big(\sum_{n}|a_{j,\n}|\big)^2 <
+\infty$, $M_\theta f$ is defined for every $\theta$, the function
$\theta \to \|M_\theta\|_2^2$ is continuous and is equal everywhere
to $\varphi_f$. For a general function $f \in L_0^2(G)$, it is
defined for $\theta$ in a set $\Lambda_0(f)$ of full measure in
$\T^d$.

Remark that the choice of the system $(\psi_j)$ generating the
orthonormal basis $(T^{\n} \psi_j)$ is not unique, so that the
definition of $M_\theta f$ is not canonical. But for algebraic
automorphisms of a compact abelian group $G$, Fourier analysis gives
a natural choice for the basis.

\vskip 3mm {\bf Approximation by orthogonal increments}

The rotated sums of $M_\theta f$ approximate the rotated sums
of $f$ in the following sense:
\begin{lem} \label{approxMart} Let $(R_n)$ be a summation sequence with associated kernel $(\tilde R_n)$.
\hfill \break a) Let $\Cal L$ be a space of functions on $\T^d$ with
the property that if $0 \leq \varphi \in \Cal L$ and $|\psi|^2 \leq
\varphi$, then $\psi, |\psi|^2 \in \Cal L$. Suppose that, for every
$\varphi \in \Cal L$, $\lim_n (\tilde R_n * \varphi)(t) =
\varphi(t)$ for a.e. $t \in \T^d$. Then, if $f$ in $L_0^2(\mu)$ is
such that $\varphi_f \in \Cal L$, we have for $\theta$ in a set
$\Lambda(f)\subset \Lambda_0(f)$ of full Lebesgue measure in $\T^d$:
\begin{eqnarray}
\lim_n {\| \sum_{\el \in \N^d} R_n(\el) \, e^{2\pi
i \langle \el, \theta \rangle} \, T^\el(f -M_\theta f)\|_2^2
\over \sum_{\el \in \N^d} |R_n(\el)|^2}
= 0. \label{approxtheta2}
\end{eqnarray}

b) If the functions $\varphi_f$, $\sum_j |\gamma_j|^2$, $\gamma_j$,
for all $j$ in $J$, are continuous and if $\varphi_f(\theta) =\sum_j
|\gamma_j(\theta)|^2 \, \forall \theta$, then, for any F\o{}lner
sequence $(R_n)$, (\ref{approxtheta2}) holds for every $\theta$.
\end{lem}
\begin{proof} The proof of a) is analogous to that of Proposition 1.4 in \cite{CohCo12}.
The projection of $f - M_\theta f$ on $\cH_j$ is $f_j
- \gamma_j(\theta) \psi_j$ and its spectral density is
$$\varphi_{f - M_\theta f}(t) =\sum_{j\in J}|\gamma_j(t)-\gamma_j(\theta)|^2
= \sum_j |\gamma_j(t)|^2 + \sum_j |\gamma_j(\theta)|^2 - \sum_j
\gamma_j(t) {\overline \gamma_j(\theta)} - \sum_j {\overline
\gamma_j(t)} \gamma_j(\theta).$$
We have
\begin{eqnarray}
{\| \sum_{\el \in \N^d} R_n(\el) \, e^{2\pi
i \langle \el, \theta \rangle} \, T^\el(f -M_\theta f)\|_2^2
\over \sum_{\el \in \N^d} |R_n(\el)|^2}  &=& \int_{\T^d}
\tilde R_n(t-\theta) \ \varphi_{f - M_\theta f}(t) \ dt. \label{quadAppr1}
\end{eqnarray}

Observe that $|\gamma_j|^2 \leq \varphi_f$. Let $\Lambda_0'(f)$ be the set of
full measure of $\theta$'s given by the hypothesis such that convergence holds at $\theta$ for
$|\gamma_j|^2$, $\gamma_j$, $ \forall j\in J$, and $\sum_{j \, \in
J} |\gamma_j|^2$.

Take $\theta \in \Lambda_0(f) \cap \Lambda_0'(f)$. Let $\varepsilon >0$
and let $J_0=J_0(\varepsilon,\theta)$ be a finite subset of $J$ such
that $\sum_{j \not \in J_0}|\gamma_j(\theta)|^2 < \varepsilon$.
Since
\begin{eqnarray*}
&&\lim_n \int_{\T^d} \Tilde R_n(t-\theta) \ \sum_{j \not \in J_0} \big
|\gamma_j(t)|^2 \ dt \\ &=& \lim_n \int_{\T^d} \Tilde R_n(t-\theta) \
\sum_{j \in J} |\gamma_j(t)|^2 \ dt - \lim_n \int_{\T^d}
\Tilde R_n(t-\theta) \ \sum_{j \, \in J_0} |\gamma_j(t)|^2 \ dt = \sum_{j
\not \in J_0} |\gamma_j(\theta)|^2,
\end{eqnarray*}
we have
\begin{eqnarray*}
&&\limsup_n \int_{\T^d} \Tilde R_n(t-\theta) \ \varphi_{f - M_\theta f}(t) \ dt\\
&& \leq \lim_n \int_{\T^d} \Tilde R_n(t-\theta) \ \sum_{j \, \in J_0}
\big[|\gamma_j(t)|^2 + |\gamma_j(\theta)|^2 - \gamma_j(t) {\overline
\gamma_j(\theta)} - {\overline \gamma_j(t)} \gamma_j(\theta)\big] \
dt \\ &&\quad \quad + 2 \lim_n \int_{\T^d} \Tilde R_n(t-\theta) \ \sum_{j
\not \in J_0}
\big[|\gamma_j(t)|^2 + |\gamma_j(\theta)|^2\big] \ dt\\
&& = \sum_{j \, \in J_0} \big[|\gamma_j(\theta)|^2 +
|\gamma_j(\theta)|^2 - \gamma_j(\theta) {\overline \gamma_j(\theta)}
- {\overline \gamma_j(\theta)} \gamma_j(\theta)\big] + 4 \sum_{j
\not \in J_0} |\gamma_j(\theta)|^2 \leq 0 + 4 \varepsilon.
\end{eqnarray*}
Therefore $\Lambda(f)$, the set for which (\ref{approxtheta2}) holds,
contains $\Lambda_0(f) \cap \Lambda_0' (f)$ and has full measure.

The proof of b) uses the same expansion as in 1).
\end{proof}

\vskip 2mm {\bf Variance for summation sequences}

Let $(D_n) \subset \N^d$ be an increasing sequence of subsets. For
$f \in L_0^2(\mu)$, the asymptotic variance at $\theta$ along
$(D_n)$ is, when it exists, the limit
\begin{eqnarray}
\sigma_\theta^2(f) = \lim_n { \|\sum_{{\el} \, \in D_n} \, e^{2\pi i
\langle \el, \theta \rangle} T^{{\el}} f\|_2^2 \over |D_n|}.
\label{VarergRotSumDn}
\end{eqnarray}

By the spectral theorem, if $\varphi_f \in L^1(\T^d)$ is the
spectral density of $f$ and $\tilde R_n$ the kernel associated to
$(D_n)$, then
\begin{eqnarray}
|D_n|^{-1} \, \|\sum_{{\el} \, \in D_n} \, e^{2\pi i \langle \el,
\theta \rangle} T^{{\el}} f\|_2^2 = (\tilde R_n *
\varphi_f)(\theta). \label{ConvVarRotergSumDn}
\end{eqnarray}
If $(D_n)$ is a sequence of $d$-dimensional cubes, we obtain, when it exists,
the usual asymptotic variance at $\theta$. By the  Fej\'er-Lebesgue
theorem, for of cubes, for every  $f$ in $\Cal H$ it exists and is equal to
$\varphi_f(\theta)$ for a.e. $\theta$.

When $\varphi_f$ is continuous, for F\o{}lner sequences, for every
$\theta$, the asymptotic variance at $\theta$ is
$\varphi_f(\theta)$. More generally, if $(R_n)$ is a regular
summation sequence with $(\tilde R_n)$ weakly converging to the
measure $\zeta$ on $\T^d$, the asymptotic variance at $\theta$ is
$\int_{\T^d} \varphi_f(\theta - t) \, d\zeta(t)$.

\vskip 2mm {\bf Variance for barycenters}

Let $P$ and $P_\theta$ be defined by (\ref{defBary}) for $d$
commuting unitary operators $T_1, ..., T_d$ on a Hilbert space
${\Cal H}$ generating a group ${\Cal S}$ with the Lebesgue spectrum
property and let $(p_1, ..., p_d)$ be a probability vector such that
$p_j > 0, \forall j$. If $\varphi_f$ is the spectral density of $f$
in ${\Cal H}$ with respect to the action of ${\Cal S}$, we have:
$$\|P_\theta^n f\|_2^2 = \int_{\T^d} \, |\sum_{j=1}^d \, p_j \,
e^{2\pi i t_j}|^{2n} \,\varphi_f(\theta_1-t_1, ..., \theta_d-t_d) \
dt_1 \, ... \, dt_d.$$

In order to find the normalization of $P^n f$ for $f \in \Cal H$, we
need an estimation, when $n \to \infty$, of the integral
$I_n:=\int_{\T^d} |\sum_j p_j e^{2\pi i t_j}|^{2n} \ dt_1 ... dt_d$.

\begin{prop} \label{PropEstimIn} If $(p_1, ..., p_d)$ is
a probability vector such that $p_j > 0, \forall j$, we have
\begin{eqnarray}
&&\lim_n n^{d-1 \over 2} \int_{\T^d} |\sum_j p_j e^{2\pi i
t_j}|^{2n} \ dt_1 ... dt_d = (4\pi)^{-{d-1 \over 2}} \, (p_1 ...
p_{d})^{-\frac12}. \label{EstimIn0}
\end{eqnarray}
\end{prop}

\begin{lem} \label{quadrF} Let $r$ be an integer $\geq 1$ and let $(q_1, ...,
q_r)$ be a vector such that $q_j > 0, \forall j$ and $\sum_j q_j
\leq 1$. Then the quadratic form $Q$  on $\R^{r}$ defined by
\begin{eqnarray}
&&Q(\t) = \sum_{j=1}^{r} q_{j} t_j^2 - (\sum_{j=1}^{r} q_{j} t_j)^2
\label{quadrQ}
\end{eqnarray}
is positive definite with determinant $(1- \sum_j q_j) \,q_1 ...
q_r$.
\end{lem}
\begin{proof} The proof is by induction on $r$. Let us consider
the polynomial in $t_1$ of degree 2:
$$q_1 t_1^2 + \sum_2^r q_j t_j^2 - (q_1 t_1 + \sum_2^r q_j t_j)^2 =
(q_1 - q_1^2) t_1^2 - 2 q_1 (\sum_2^r q_j t_j) t_1 + \sum_2^r q_j
t_j^2 - (\sum_2^r q_j t_j)^2.$$

It is always $\geq 0$, since its discriminant
$$q_1^2 (\sum_2^r q_j t_j)^2 - q_1 (1 - q_1) (\sum_2^r q_j t_j^2 -
(\sum_2^r q_j t_j)^2) = q_1(1-q_1)^2 [(\sum_2^r {q_j \over 1 - q_1}
t_j)^2 - \sum_2^r {q_j \over 1 - q_1}t_j^2],$$ is $< 0$ for
$\sum_{j=2}^r t_j^2 \not = 0$ by the induction hypothesis, since
${q_j \over 1 - q_1} > 0$ and $\sum_{j=2}^r {q_j \over 1 - q_1} \leq
1$.

The quadratic form is given by the symmetric matrix: $A = \diag(q_1,
..., q_r) \, \, B$, where
\begin{eqnarray*} B &&=  \left( \begin{matrix} 1 - q_1 & -q_2 & . & -q_r
\cr -q_1 & 1- q_2 & . & -q_r \cr . & . & . & . \cr -q_1 & -q_2& . &
1 -q_r \cr
\end{matrix} \right).
\end{eqnarray*}

The determinant of $B$ is of the form $\alpha + \sum_j \beta_j q_j$,
where the coefficients $\alpha, \beta_1, ..., \beta_r$ are constant.
Giving to $q_1, ..., q_r$ the values $0$ except for one of them, we
find $\alpha =1$, $\beta_1 = \beta_2 = ... = \beta_r = -1$. Hence
$\det A = (1- \sum_j q_j) \,q_1 ... q_r$.
\end{proof}

Remark that the positive definiteness follows also from the
properties of $F$, since $Q$ gives the approximation of $F$ defined
below at order 2.

{\it Proof of Proposition \ref{PropEstimIn}} \
Since $|\sum_j p_j e^{2\pi i t_j}|^{2n}= |p_1 + \sum_{j=2}^d p_j e^{2\pi i
(t_j - t_1)}|^{2n}$, we have $I_n = \int_{\T^{d-1}} |p_1 + \sum_{j=2}^d p_j
e^{2\pi i t_j}|^{2n} \ dt_2 ... dt_{d}$.

Putting $q_j:=p_{j+1}$, $j=1, ..., d-1$, $r=d-1$, we have $q_j > 0$,
$\sum_1^r q_j < 1$. With the notation $F(\t) := 1 - |1 +
\sum_{j=1}^r q_j (e^{2\pi i t_j} -1)|^2$ the computation reduces to
estimate:
$$I_n:=\int_{\T^r} [|1 + \sum_{j=1}^r q_j (e^{2\pi i t_j} -1)|^2]^{n} \ dt_1 ... dt_r
=\int_{\T^r} [1- F(\t)]^{n} \ dt_1 ... dt_r.$$ A point $\t=(t_1,
...t_{r})$ of the torus is represented by coordinates such that:
$-\frac12 \leq t_j < \frac12$. We have $F(\t) \geq 0$ and $F(\t)= 0$
if and only if $\t = 0$. Let us prove the stronger property: there
is $c > 0$ such that
\begin{eqnarray}
F(\t) \geq c \|\t\|^2, \ \forall \t: -\frac12 \leq t_j < \frac12.
\label{neighbF}
\end{eqnarray}
Indeed Inequality (\ref{neighbF}) is clearly satisfied outside a small open
neighborhood $V$ of $0$, since $F(\t)$ is bounded away from 0 for
$\t$ in $V$. On $V$, we can replace $F$ by a positive definite quadratic form
as we will see below. This shows the result on $V$.

From the convergence
$$\lim_n n^{{r \over 2}} \int_{\{\t \in \T^{r}: \|\t\| > {\ln n
\over \sqrt{n}} \}} (1-F(\t))^{n} \ d\t \leq \lim_n n^{{d-1 \over
2}} (1 - c{(\ln n)^2 \over {n}})^n = 0,$$ it follows, with $J_n := \int_{\{\t \in \T^{r}:
\|\t\| \leq {\ln n \over \sqrt{n}} \}} (1 -F(\t))^{n} \ d\t$:
$$\lim_n n^{{r \over 2}} \int_{\t \in \T^{d-1}} (1-F(\t))^{n} \ d\t
= \lim_n n^{{r \over 2}} J_n.$$
By taking the Taylor approximation of order 2 at 0 of the
exponential function $e^{i t_j} = 1 + i t_j - {t_j^2 \over 2} +  i
\gamma_1(t_j) +\gamma_2(t_j)$, with $|\gamma_1(t_j)+ |\gamma_2(t_j)|
= o(|t_j|^2)$, we obtain:
$$F(t) = Q(2\pi t) + \gamma(t), \ \text{ with } Q(\t) = \sum q_j t_j^2
- (\sum q_j t_j)^2 \ \text{ and } \gamma(t) = o(\|t\|^2).$$

The quadratic form $Q$ is the form defined by (\ref{quadrQ}).
Therefore, it is positive definite by Lemma \ref{quadrF} and there
is $c > 0$ such that $Q(\t) \geq c \|\t\|^2, \forall \t \in \R^{r}$.

We have $\lim_{\delta \downarrow 0} \sup_{\|t\| \leq \delta} F(t) /
Q(2\pi t) = 1$. With the notation $\u=(u_1, ..., u_r), \, \t=(t_1,
..., t_r)$ and the change of variable $\u = \sqrt n \, \t$, we get:
$$n^{{r \over 2}} J_n \sim \int_{\{\|\u\| \leq \ln n \}}
(1 - Q({2\pi \u \over \sqrt n}))^{n} \, d \u \to {1 \over
(2\pi)^{r}} \,  \int_{\R^{r}} e^{- Q(u)} \, d\u.$$

We have $\int_{\R^{r}} e^{- Q(u)} \, d\u = \pi^{{r \over 2}}
\det(A)^{-\frac12} = \pi^{r \over 2} \, (p_1... p_d)^{-\frac12}$.
Therefore we obtain:
\begin{eqnarray*}
&&\lim_n n^{d-1 \over 2} \int_{\T^d} |\sum_j p_j e^{2\pi i
t_j}|^{2n} \ dt_1 ... dt_d = (4\pi)^{-{r \over 2}} \, (p_1 ...
p_{d})^{-\frac12}.
\end{eqnarray*}
\eop

\vskip 2mm {\it Example}: With $K_n(t_1,t_2) := \sqrt{\pi
n}|({e^{2\pi i t_1} + e^{2\pi i t_2} \over 2})|^{2n}$, we have
$\int_{\T^2}K_n(t_1,t_2) dt_1 dt_2 \to 1$. This can be shown also
using Stirling's approximation:
\begin{eqnarray*}
\int_{\T^2}K_n(t_1,t_2)dt_1dt_2 = \frac{\sqrt{\pi
n}}{4^n}\,\sum_{k=0}^n{n\choose k}^2= \frac{\sqrt{\pi n}}{4^n}\,
{2n\choose n}\underset{n\to\infty}\to 1.
\end{eqnarray*}

\begin{prop} \label{barycenter} If $\varphi_f$ is continuous, then
for every $\theta\in\T^d$ we have
\begin{eqnarray}
\lim_{n\to\infty} \, (4\pi)^{{d-1 \over 2}} \, (p_1 ...
p_{d})^{\frac12} \, n^{d-1 \over 2} \,\|P_\theta^n f\|_2^2 = \int_\T
\varphi_f (\theta_1+u, ..., \theta_d+u) \, du. \label{convKn}
\end{eqnarray}
\end{prop}
\begin{proof} Let us put $c_n:= (4\pi)^{{d-1 \over
2}} \, (p_1 ... p_{d})^{\frac12} \, n^{d-1 \over 2}$ for the
normalization coefficient and $$K_n(t_1, ..., t_d) := c_n \,
|\sum_{j=1}^d \, p_j \, e^{2\pi i t_j}|^{2n}.$$

We have $c_n \, \|P_\theta^n f\|_2^2= (K_n*\varphi_f)(\theta_1, ...,
\theta_d)$ and, by (\ref{EstimIn0}) $\int_{\T^d} K_n(t_1, ..., t_d)
\, dt_1 ... dt_d \to 1$.

Let us show that for $\varphi$ continuous on $\T^d$, $\lim_n
\int_{\T^d} K_n \, \varphi \, dt_1... dt_d =  \int_\T \varphi(u,
..., u) \, du$. Using the density of trigonometric polynomials for
the uniform norm, it is enough to prove it for characters
$\chi_{\k}(\t)=\e^{2\pi i \sum_j k_j t_j}$, i.e., to prove that for
$\varphi = \chi_{\k}$ the limit is 0 if $\sum_\ell k_\ell \not = 0$,
and 1 if $\sum_\ell k_\ell = 0$. We have
\begin{eqnarray*}
&&\int_{\T^d} K_n(t_1, ..., t_d) \, e^{2\pi i \sum_\ell k_\ell
t_\ell} \, dt_1 ... dt_d  = c_n \, \int_{\T^d} |\sum_{j=1}^d \, p_j
\, e^{2\pi i t_j}|^{2n} \, e^{2\pi i \sum_\ell k_\ell t_\ell} \, dt_1 ... dt_d \\
&&= (c_n \, \int_{\T^{d-1}} |p_1 + \sum_{j=2}^d \, p_j \, e^{2\pi i
(t_j - t_1)}|^{2n} \, e^{2\pi i \sum_{\ell=2}^d k_\ell (t_\ell -
t_1)} \, dt_2 ... dt_d)  \ \int_{\T} e^{2\pi i (\sum_\ell k_\ell)
t_1} dt_1.
\end{eqnarray*}

Therefore it remains to show that the limit of the first factor when
$n \to \infty$ is 1. Using the proof and the result of Proposition
\ref{PropEstimIn}, we find that this factor is equivalent to
$$(4\pi)^{{d-1 \over 2}} \, (p_1 ... p_{d})^{\frac12}
\int_{\{\|\u\| \leq \ln n \}} (1 - Q({2 \pi \u \over \sqrt
n}))^n e^{2\pi i \sum_1^r k_{\ell +1} {u_\ell \over \sqrt n}} \
d\u,$$ which tends to 1.
\end{proof}

\vskip 3mm
\subsection{\bf Nullity of variance and coboundaries}
 \

Let $\Cal H$ be a Hilbert space and let $T_1$ and $T_2$ be two
commuting unitary operators acting on $\Cal H$. Assuming the
Lebesgue spectrum property for the $\Z^2$-action generated by $T_1$
and $T_2$, we study in this subsection the degeneracy of the
variance. Here we consider, for simplicity, the case of two unitary
commuting operators, but the results are valid for any finite family
of commuting unitary operators.

\vskip 3mm {\bf Single Lebesgue spectrum}

At first, let us assume that there is $\psi \in \Cal H$ such that
the family of vectors $T_1^k T_2^r \psi$ for $(k,r) \in \Z^2$ is an
orthonormal basis of $\Cal H$ (simplicity of the spectrum).

\begin{lem} \label{cobRepr} Let $f$ be in $\Cal H$ and $f= \sum_{(k,r)
\in \Z^2} a_{k,r} \, T_1^k T_2^r \psi$ be the representation of $f$
in the orthonormal basis $(T_1^k T_2^r \psi, \ (k,r) \in \Z^2)$. If
\begin{eqnarray}
A:=\sum_{k,r \in \Z^2} \, (1 + |k| + |r|)\, |a_{k,r}| < +\infty,
\label{convCob}
\end{eqnarray}
there exists $u, v \in \Cal H$ with $\|u\|, \|v\| \leq A$ such that
$$f = (\sum_{(k,r) \in \Z^2} a_{k,r}) \, \psi + (I- T_1) u + (I-T_2)v.$$
If $\sum_{(k,r) \in \Z^2} a_{k,r} = 0$, then $f$ is sum of two
coboundaries respectively for $T_1$ and $T_2$:
$$f = (I- T_1) u + (I-T_2)v.$$
\end{lem}
\begin{proof} \ 1) We start with a formal computation.
Let us decompose $f$ into vectors whose coefficients are supported
on disjoint quadrants of increasing dimensions. If $f = \sum_{k,r
\in \Z^2} a_{k,r}T_1^k T_2^r \psi$, we write
\begin{eqnarray}
f = f_{0,0} + f_{1,0} + f_{0,1} + f_{-1,0} + f_{0,-1} + f_{1,1} +
f_{-1,1} + f_{1,-1}+ f_{-1,-1}, \label{decQuadr}
\end{eqnarray}
with
\begin{eqnarray*}
&&f_{0,0} = a_{0,0} \psi, \ f_{1,0}=\sum_{k > 0} a_{k,0} \,T_1^k
\psi,
\ f_{0,1} = \sum_{r > 0} a_{0,r} \, T_2^r \psi, \\
&&f_{-1,0} =\sum_{k > 0} a_{-k,0} \,T_1^{-k} \psi, \ f_{0,-1} =
\sum_{r > 0} a_{0,-r} \, T_2^{-r} \psi, \\
&&f_{1,1}= \sum_{k,r > 0} a_{k,r} \,T_1^k T_2^r \psi, \  f_{-1,1}=
\sum_{k,r > 0} a_{-k,r} \,T_1^{-k} T_2^r \psi, \\
&&f_{1,-1}= \sum_{k,r > 0} a_{k,-r} \,T_1^{k} T_2^{-r} \psi, \
f_{-1,-1}= \sum_{k,r > 0} a_{-k,-r} \,T_1^{-k} T_2^{-r} \psi.
\end{eqnarray*}
For each component given by a quadrant, we solve the corresponding
coboundary equation up to $constant \times \psi$.

\vskip 3mm With $f$ decomposed as in (\ref{decQuadr}), the
components can be formally written in the following way, with
$\varepsilon_i, \varepsilon_i' \in \{0, +1, -1\}$, for $i=1, 2$:
\begin{eqnarray*}
f_{0,0} &=& u_{0,0} = a_{0,0} \, \psi,  \ \ f_{\varepsilon_1,0} =
u_{\varepsilon_1,0}^0 + (T_1^{\varepsilon_1}-I)u_{\varepsilon_1,
\,0}^{\varepsilon_1, \,0}, \ \ f_{0,\varepsilon_2} = u_{0,
\varepsilon_2}^0 +
(T_2^{\varepsilon_2}-I)u_{0, \, \varepsilon_2}^{0, \, \varepsilon_2}, \\
f_{\varepsilon_1,\varepsilon_2} &=&
u_{\varepsilon_1,\varepsilon_2}^0 + (T_1^{\varepsilon_1}-I)
u_{\varepsilon_1,\varepsilon_2}^{\varepsilon_1, 0}+
(T_2^{\varepsilon_2}-I) u_{\varepsilon_1,\varepsilon_2}^{0,
\varepsilon_2} - (T_1^{\varepsilon_1}-I) (T_2^{\varepsilon_2}-I)
u_{\varepsilon_1,\varepsilon_2}^{\varepsilon_1, \varepsilon_2},
\end{eqnarray*}
where
\begin{eqnarray*}
u_{\varepsilon_1,0}^0 &=&  (\sum_{t \geq 1} a_{\varepsilon_1 t,\,0})
\, \psi, \ \  u_{\varepsilon_1, \,0}^{\varepsilon_1, \,0} =
\sum_{k\geq 0} (\sum_{t \geq k+1} a_{\varepsilon_1 t, \,0})
T_1^{\varepsilon_1 k}\, \psi, \\  u_{0, \varepsilon_2}^0 &=&
(\sum_{s \geq 1} a_{0, \, \varepsilon_2 s}) \, \psi, \ \  u_{0, \,
\varepsilon_2}^{0, \, \varepsilon_2} = \sum_{r\geq 0} (\sum_{s \geq
r+1} a_{0, \, \varepsilon_2 s})
T_2^{\varepsilon_2 r}\, \psi,\\
u_{\varepsilon_1,\varepsilon_2}^0 &=& (\sum_{t,s \geq 1}
a_{\varepsilon_1 t, \, \varepsilon_2 s}) \, \psi, \
u_{\varepsilon_1,\varepsilon_2}^{\varepsilon_1, 0} = \sum_{k 0,r
\geq 1} (\sum_{t \geq k+1} a_{\varepsilon_1 t, \, \varepsilon_2 r})
\, T_1^{\varepsilon_1 k} T_2^{\varepsilon_2 r}\, \psi, \\
u_{\varepsilon_1,\varepsilon_2}^{0,\varepsilon_2} &=& \sum_{k\geq
1,r \geq 0} (\sum_{s \geq r+1} a_{\varepsilon_1 k, \, \varepsilon_2
s}) \, T_1^{\varepsilon_1 k} T_2^{\varepsilon_2 r} \, \psi,\\
u_{\varepsilon_1,\varepsilon_2}^{\varepsilon_1, \varepsilon_2} &=&
\sum_{k,r \geq 0} (\sum_{t \geq k+1, s \geq r+1} a_{\varepsilon_1 t,
\, \varepsilon_2 s}) \, T_1^{\varepsilon_1 k} T_2^{\varepsilon_2
r}\, \psi.
\end{eqnarray*}

More explicitly we have, for instance,
\begin{eqnarray*}
&&f_{1,0} = u_{1,0}^0 + (T_1-I)u_{1,0}^{1,0} = (\sum_{t \geq 1}
a_{t,0}) \, \psi +(T_1-I) \, [\sum_{k\geq 0} (\sum_{t \geq k+1} a_{t,0}) T_1^{k}\, \psi ], \\
&&f_{1,1} =  u_{1,1}^0 + (T_1-I)u_{1,1}^{1,0} + (T_2-I) u_{1,1}^{0,1} +(T_1-I)(T_2-I)u_{1,1}^{1,1}\\
&&= (\sum_{t,s \geq 1} a_{t,s}) \, \psi +(T_1-I) \, [\sum_{k\geq
0,r \geq 1} (\sum_{t \geq k+1} a_{t,r}) \, T_1^k T_2^r\, \psi] \\
&&+ (T_2-I) \, [\sum_{k\geq 1,r \geq 0} (\sum_{s \geq r+1} a_{k,s})
\, T_1^k T_2^r\, \psi] - (T_1-I) (T_2-I) \, [\sum_{k,r \geq 0}
(\sum_{t \geq k+1, s \geq r+1} a_{t,s}) \, T_1^k T_2^r\, \psi].
\end{eqnarray*}

By summing the previous expressions, we obtain the following
representation of $f$:
\begin{eqnarray*}
f = (\sum a_{t,s}) \, \psi \ &+ (T_1-I)(u_1^1 - T_1^{-1}u_{-1}^1 +
u_1^{1,2} - T_1^{-1} u_{-1}^{-1,2} + u_{1}^{1,-2} - T_1^{-1}u_{-1}^{-1,-2})\\
&\ + (T_2-I)(u_2^1 - T_2^{-1}u_{-2}^1 + u_2^{1,2} -
T_2^{-1}u_{-2}^{1,-2}
+ u_{2}^{-1,2} - T_2^{-1}u_{-2}^{-1,-2})\\
& \ + (T_1-I)(T_2-I)(u_{1,2}^{1,2} - T_1^{-1}u_{-1,2}^{-1,2} -
T_2^{-1} u_{1,-2}^{1,-2} + T_1^{-1}T_2^{-1}u_{-1,-2}^{-1,-2}).
\end{eqnarray*}

The first term is the vector $a(f) \, \psi$, where $a(f)$ is the
constant $\sum_{k,r \in \Z^2} a_{k,r}$ obtained as the sum $u_0^0 +
u_1^0 + u_2^0 + u_{-1}^0 + u_{-2}^0 +  u_0^{1,2} + u_0^{-1,2} +
u_0^{1,-2} + u_0^{-1,-2}$. The second term is a sum of coboundaries.
If $a(f) = 0$, then $f$ reduces to a sum of coboundaries.

2) Now we examine the question of convergence in the previous
computation. We need the convergence of the following series (for
$\varepsilon_1, \varepsilon_2 = \pm 1$):
\begin{eqnarray*}
&& \sum_{t,s \geq 1} a_{\varepsilon_1 t, \, \varepsilon_2 s}, \
\sum_{t \geq k+1} a_{\varepsilon_1 t, \, \varepsilon_2 r}, \ \sum_{s
\geq r+1} a_{\varepsilon_1 k, \, \varepsilon_2 s}, \ \sum_{t \geq
k+1, s \geq r+1} a_{\varepsilon_1 t, \, \varepsilon_2 s}, \\
&&\sum_{k\geq 0,r \geq 1} \ |\sum_{t \geq k+1} a_{\varepsilon_1 t,
\, \varepsilon_2 r}|^2, \ \sum_{k\geq 1,r \geq 0} |\sum_{s \geq r+1}
a_{\varepsilon_1 k, \, \varepsilon_2 s}|^2,  \ \sum_{k,r \geq 0}
|\sum_{t \geq k+1, s \geq r+1} a_{\varepsilon_1 t, \, \varepsilon_2
s}|^2.
\end{eqnarray*}

Sufficient conditions for the convergence are:
\begin{eqnarray*}
&& \sum_{k,r \in \Z^2} |a_{k, \, r}| < +\infty, \ \sum_{k\geq 0,r
\geq 1} \ (\sum_{t \geq k+1} |a_{\varepsilon_1 t, \, \varepsilon_2
r}|)^2 < +\infty, \\ &&\sum_{k\geq 1,r \geq 0} (\sum_{s \geq r+1}
|a_{\varepsilon_1 k, \, \varepsilon_2 s}|)^2 < +\infty,  \ \
\sum_{k,r \geq 0} (\sum_{t \geq k+1, s \geq r+1} |a_{\varepsilon_1
t, \, \varepsilon_2 s}|)^2 < +\infty.
\end{eqnarray*}

We have:
\begin{eqnarray*}
&&\sum_{k\ge 0, r \ge 0} (\sum_{t \geq k, s \geq r} |a_{t,s}|)^2 =
\sum_{k\ge 0, r \ge 0} (\sum_{t, t' \geq k, \, s, s' \geq r}
|a_{t,s}| |a_{t',s'}|) \\
&&\leq \sum_{t,t'\geq 0,s,s'\geq 0} |a_{t,s}| |a_{t',s'}| \sum_k
1_{0 \leq k \leq \inf(t, t')} \sum_r 1_{0 \leq r \leq \inf(s, s')} \\
&&= \sum_{t,t' \geq 0,s,s' \geq 0} \,  |a_{t,s}| |a_{t',s'}| \, (1 +
\inf(t, t'))\, (1+\inf(s, s')) \\
&&\leq \sum_{t,t',s,s' \geq 0} |a_{t,s}| |a_{t',s'}| (1+ t + s) \,
(1+ t' + s') = (\sum_{t\geq 0,s\geq 0} (1+ t + s) \, |a_{t,s}|)^2.
\end{eqnarray*}
An analogous bound is valid for the indices with $\pm$ signs.
Therefore, convergence holds if (\ref{convCob}) is satisfied and we
get $\sum_{\t \in \Z^2} \, (1 + \|\t\|)\, |a_{\t}|$ as a bound for
the norm of the vectors $u_{\varepsilon_1,0}^0, u_{\varepsilon_1,
\,0}^{\varepsilon_1, \,0}, u_{0, \varepsilon_2}^0, u_{0, \,
\varepsilon_2}^{0, \, \varepsilon_2},
u_{\varepsilon_1,\varepsilon_2}^0,
u_{\varepsilon_1,\varepsilon_2}^{\varepsilon_1, 0},
u_{\varepsilon_1,\varepsilon_2}^{0,\varepsilon_2},
u_{\varepsilon_1,\varepsilon_2}^{\varepsilon_1, \varepsilon_2}$.
\end{proof}

\vskip 3mm \goodbreak
{\bf Countable Lebesgue spectrum}

We suppose now that the action on $\Cal H$ has a countable Lebesgue
spectrum: there exists a countable set $(\psi_j, j \in J)$ in $\Cal
H$ such that the family of vectors $\{T_1^k T_2^r \psi_j,j \in J,
(k,r) \in \Z^2\}$ is an orthonormal basis of $\Cal H$. The
representation of $f$ in this orthonormal basis is given by $f =
\sum_j f_j = \sum_{j \in J} \, (\sum_{(k,r) \in \Z^2} a_{j, (k,r)}
\, T_1^k T_2^r \psi_j)$, with $a_{j, (k,r)} = \langle f, \, T_1^k
T_2^r \psi_j\rangle$. Recall that
\begin{eqnarray*}
&&M_\theta(f) = \sum_{j \in J} \gamma_j(\theta) \psi_j, \text{ with
} \gamma_j(\theta)= \sum_{\k \in \Z^2}  a_{j, \k} \, e^{2 \pi i
\langle \k, \theta \rangle}.
\end{eqnarray*}

Using Lemma \ref{cobRepr}, we have under a convergence condition:
\begin{eqnarray*}
f_j &=& \gamma_j(\theta) \psi_j + (I- e^{2\pi i \theta_1} T_1) u_{j,
\theta} + (I- e^{2\pi i \theta_2}T_2) v_{j, \theta}, \forall j \in
J, \\
f&=& M_\theta(f) + (I- e^{2\pi i \theta_1}T_1) \sum_{j \in J} u_{j,
\theta} + (I- e^{2\pi i \theta_2} T_2) \sum_{j \in J} v_{j, \theta}.
\end{eqnarray*}

The result for $d$ generators is the following:
\begin{lem} \label{cobCond24} Suppose that the following condition is
satisfied:
\begin{eqnarray}
\sum_j \sum_{\k \in \Z^d} (1 + \|\k\|^d) \, |a_{j, \k}| < \infty.
\label{condConv2}
\end{eqnarray}
Then there are $v, u_1, ..., u_d \in \Cal H$ such that the family
$\{T^\n v, \n \in \Z^d\}$ is orthogonal and
\begin{eqnarray}
f= v + \sum_{t=1}^d (I- T_t) u_t. \label{cobound10A}
\end{eqnarray}
The variance is 0, if and only $f$ is a mixed coboundary.

For every $\theta$, the rotated variance $\sigma_\theta^2(f)$ is
null if and only if there are $u_{t,\theta} \in \Cal H$, for $t=1,
..., d$,  such that $f= \sum_{t=1}^d (I- e^{2\pi i \theta_t} T_t)
u_{t,\theta}$.

In the topological framework, when the $T^\n \psi_j$'s are
continuous and uniformly bounded with respect to $\n$ and $j$, then
the functions $v$ and $u_t$ are continuous.
\end{lem}

\section{\bf Multidimensional actions by endomorphisms} \label{endoGroup}

In what follows we consider a finitely generated semigroup ${\Cal
S}$ of surjective endomorphisms of $G$, a compact abelian group with
Haar measure denoted by $\mu$. The group of characters of $G$ will
be denoted by $\hat G$ (or $H$) and the set of non trivial
characters by $\hat G^*$ (or $H^*$).

After choosing a system $A_1, ..., A_d$ of generators, every element
in $\Cal S$ can be represented as $A^\n$ with the notation $A^\n:=
A_1^{n_1}... A_d^{n_d}$, $\n = (n_1, ..., n_d) \in \N^d$, and we
obtain an action $\n \to A^\n$ of $\N^d$ by endomorphisms on $G$.

If $f$ is function on $G$, $A^\n f$ stands for $f \circ A^\n$. We
use also the notation $T^\n f$. For $T \in \Cal S$, we denote by the
same letter its action on $G$ and on the dual $H$ of $G$. The
Fourier coefficients of a function $f$ in $L^2(G)$ are $c_f(\chi) :=
\int_G \, \overline \chi \, f \, d\mu$.

Every surjective endomorphism $A$ of $G$ defines a measure
preserving transformation on $(G, \mu)$ and a dual endomorphism on
$\hat G$, the group of characters of $G$. For simplicity, we use the
same notation for the action on $G$ and on $\hat G$.

The first subsections are preparatory for the CLT.

\vskip 3mm \subsection{\bf Preliminaries} \label{ssectPrelim}

\

\vskip 2mm {\bf Embedding of a semigroup of endomorphisms in a group}
\begin{lem} \label{embedd} Let ${\Cal S}$ be a commutative semigroup
of surjective endomorphisms on a compact abelian group $G$ with dual
group $\hat G$. There is a compact abelian group $\tilde G$ such
that $G$ is a factor of $G$ and ${\Cal S}$ is embedded in a group
$\tilde {\Cal S}$ of automorphisms of $\tilde G$. If $G$ is
connected, then $\tilde G$  is also connected.\end{lem}
\begin{proof} We construct a discrete group $\tilde
H$ such that $\hat G$ is isomorphic to a subgroup of $\tilde H$. The
group $\tilde H$ is defined as the quotient of the group $\{(\chi,
A), \chi \in \hat G, A \in \Cal S\}$ (endowed with the additive law
on the components) by the following equivalence relation $R$:
$(\chi, A)$ is equivalent to $(\chi', A')$ if $A'\chi = A\chi'$. The
transitivity of the relation $R$ follows from the injectivity of
each $A \in \Cal S$ acting on $\hat G$. The map $\chi \in \hat G \to
(\chi, Id) / R$ is injective. The elements $A \in \Cal S$ act on
$\tilde H$ by $(\chi, B) / R  \to (A\chi, B) / R$. The equivalence
classes are stable by this action. We can identify $\Cal S$ and its
image. For $A \in \Cal S$, the automorphism $(\chi, B) / R \to
(\chi, AB) / R$ is the inverse of $(\chi, B) / R \to (A\chi, B) /
R$.

We obtain an embedding of ${\Cal S}$ in a group $\tilde {\Cal S}$ of
automorphisms of $\tilde H$. If $\hat G$ is torsion free, then
$\tilde H$ is also torsion free and its dual $\tilde G$ is a
connected compact abelian group.
\end{proof}

{\it Our data will be a finite set of commuting surjective
endomorphisms $A_i$ of $G$ such that the generated group $\tilde
{\Cal S}$ is torsion-free.}

If necessary, we consider also the group of automorphisms $\tilde
{\Cal S}$ on the extension $\tilde G$. Since $\tilde {\Cal S}$ is
finitely generated and torsion-free, it has a system of $d$
independent generators (not necessarily in $\Cal S$) and it is
isomorphic to $\Z^d$. The {rank} of the action of $\Cal S$ is $d$.
The spectral analysis for  ${\Cal S}$, as for $\tilde {\Cal S}$,
takes place in $\T^d$.

Observe that, by using the projection $\pi$ from $\tilde G$ to $G$,
a function $f$ on the group $G$ can be viewed as function on $\tilde G$ and a character
$\chi \in \hat G$ as a character on $\tilde G$ via the composition
$g \to \chi(\pi g)$. Putting $\tilde f(x) = f(\pi x)$, the Fourier series of $\tilde f$ reads:
$\tilde f = \sum_{\chi \in \tilde G} c_{\tilde f}(\chi) \, \chi$. But, on an other side, we have
$f = \sum_{\chi \in G} c_{f}(\chi) \, \chi$, so that by unicity of the Fourier series,
the only non zero Fourier coefficients for $\tilde f$ are those for $\chi \in \hat G$.

It follows that, for a function $f$ defined on $G$, the computations
can be done in the group $\tilde G$ with the action of the group of
automorphisms, but expressed in terms of the Fourier coefficients of
$F$ computed in $G$.

\begin{defi} {\rm We say that a $\Z^d$-action by
automorphisms, $\n \to A^\n$, is {\it totally ergodic} if
$A_1^{n_1}... A_d^{n_d}$ is ergodic for every $\n =(n_1, ..., n_d)
\not = \underline 0$. It is equivalent to the property: $A^\n \chi
\not = \chi$, for any non trivial character $\chi$ and $\n \not =
\underline 0$. For endomorphisms, we replace the semigroup $\Cal S$
by the extension $\tilde{\Cal S}$ defined in Lemma
\ref{embedd}.}\end{defi}

What we call ``totally ergodic" in the general case of a compact
abelian group $G$ is often called ``partially hyperbolic" for
actions on a torus.

\begin{lem} \label{LebSpecEquiv}
The following conditions are equivalent for a $\Z^d$-action $T$ by
automorphisms on a compact abelian group: \hfill \break i)\  $T$ is
totally ergodic; \hfill \break ii) $T$ is $2$-mixing\footnote{\
Mixing of order 2 is defined in Definition \ref{mixLeb} (here $\Cal
H = L_0^2(G, \mu)$) or equivalently by $\lim_{\n \to \infty} \mu(B_1
\cap T^{-\n} \, B_2)= \mu(B_1) \, \mu(B_2)$, $\forall \, B_1, B_2
\in \Cal A$.}; \hfill \break iii) $T$ has the Lebesgue spectrum
property.
\end{lem}
\begin{proof} The Lebesgue spectrum property (cf. Section \ref{specAnaly}) for the
action of $\tilde {\Cal S}$ is equivalent to the fact that the
action $\tilde {\Cal S}$ on $\tilde H^*$ is free, which is total
ergodicity.

Mixing of order 2 implies total ergodicity. At last the implication
$(iii) \Rightarrow (ii)$ is a general fact.
\end{proof}

\begin{rem} Finding the dimension of $\Cal S$ and computing a set
of independent generators can be very difficult in practice. For $G=
\T^\rho = 3$, we will give explicit examples in the Appendix. Given
a finite set of commuting matrices in dimension $\rho$ with
determinant 1 for $\rho > 3$, it can be difficult and even
impossible to find independent generators via a computation.

In some cases the problem can be easier with endomorphisms. For
instance, let $p_i, i=1, ..., d$ be coprime positive integers and
$A_i: x \to q_i x \text{ mod } 1$ the corresponding endomorphisms
acting on $\T^1$. Then the $A_i$'s give a system of independent
generators of the group $\tilde {\Cal S}$ generated on the compact
abelian group dual of $\tilde \Z^\rho :=\{\k \,\Pi q_i^{\ell_i}, \k
\in \Z^\rho, \ell_i \in \Z \}$. \end{rem}

\begin{nota} {\rm  $J$ denote a {\it section} of the $\tilde {\Cal S}$ action on
$\tilde H^*$, i.e., a subset $\{\chi_j\}_{j \in J} \subset \tilde H
\stm0$ such that every $\chi \in \tilde H^*$ can be written in a
unique way as $\chi = A_1^{n_1}... A_d^{n_d} \, \chi_j$, with $j \in
J$ and $(n_1, ..., n_d) \in \Z^d$.}
\end{nota}

Using the extension, for a function $f$ in $L^2(\tilde G)$, we have
\begin{eqnarray}
\langle f, A^\n f\rangle = \sum_{\chi \in \tilde H} \ c_f(A^\n \chi) \,
\overline{c_f(\chi)}.\label{decorr1}
\end{eqnarray}
For the validity of this formula we suppose $\sum_{\chi \in \tilde
H} \ |c_f(A^\n \chi)| \, |\overline{c_f(\chi)}| < +\infty$. When $f$
is defined on $G$, the formula can be written, with $\tilde f$ the
extension of $f$ to $\tilde G$ and the convention (*) that the
coefficients $c_f(A^\n \chi)$ are replaced by 0 if $A^\n \chi \not
\in H$,
\begin{eqnarray}
\langle \tilde f, A^\n \tilde f\rangle = \sum_{\chi \in H} \
c_f(A^\n \chi) \, \overline{c_f(\chi)}.\label{2decorr1}
\end{eqnarray}

\vskip 3mm
\goodbreak
{\it Functions with absolutely convergent Fourier series}

We denote by $AC_0(G)$ the class of real functions on $G$ satisfying
$\mu(f) =0$ and with an absolutely convergent Fourier series, i.e.,
such that
\begin{eqnarray}
\|f \|_c := \sum_{\chi \in \hat G} |c_f(\chi)| < +\infty. \label{absConv}
\end{eqnarray}

\begin{thm} \label{condAbsConv} If $f$ is in $AC_0(G)$, then $\sum_{\n \in \Z^d}
|\langle A^{\n}f, f\rangle| < \infty$, the variance $\sigma^2(f)$
exists, $\sigma^2(f) = \sum_{\n \in \Z^d} \langle A^{\n}f, f\rangle$
and the spectral density $\varphi_f$ of $f$ is continuous.

Moreover, if $\Cal{N}$ is any subset of $\hat G$ and $f_1(x) =
\sum_{\chi \in \Cal{N}} c_f(\chi) \, \chi$, then
\begin{eqnarray}
\|\varphi_f\|_\infty \leq \|f-f_1\|_c^2. \label{densspectAC}
\end{eqnarray}

In particular, $\sigma(f - f_1) \leq \|f-f_1\|_c$.
\end{thm}
\begin{proof} We use the convention (*) in the notations.
By total ergodicity, for every $\chi \in \tilde H^*$, the map $\n
\in \Z^d \to A^\n \chi \in \tilde H^*$ is injective, and therefore
$\sum_{\n \in \Z^d}|c_f(A^\n \chi)| \leq \sum_{\chi \in \hat G^*}
|c_f(\chi)|$. The spectral density of $f$ is
\begin{eqnarray}
\varphi_f(t) = \sum_{j \in J} |\gamma_j(t)|^2 = \sum_{j \in J} \,
|\sum_{\n \in \Z^d \stm0} c_f(A^\n \chi_j) \, e^{2\pi i \langle\n,
t\rangle} \, |^2.\label{SpecDens}
\end{eqnarray} We have:
\begin{eqnarray*}
\|\varphi_f\|_\infty &&\leq \sum_{j \in J} \|\gamma_j\|_\infty^2
\leq
\sum_{j \in J} \, (\sum_{\n \in \Z^d} |c_f(A^\n \chi_j)|)^2 \\
&& \leq  \sum_{j \in J} \, (\sum_{\n \in \Z^d} |c_f(A^\n \chi_j)|)
(\sum_{\chi}|c_f(\chi)|) \leq (\sum_{\chi} |c_f(\chi)|)^2 = \|f
\|_{c}^2.
\end{eqnarray*}
\end{proof}

{\it Approximation by $M_\theta$}

In the algebraic setting of endomorphisms of compact abelian groups,
the family $(\psi_j)$ of the general theory (Subsection \ref{orthogIncr}) is
$(\chi_j)$. We have
\begin{eqnarray*}
&&a_{j,n} = \langle f, T^\n \chi_j\rangle = c_f(T^{\tn}\, \chi_j), \\
&&\gamma_j(\theta) = \sum_{\tn \in \Z^d} c_f(T^{\tn}\, \chi_j) \, e^{2
\pi i \langle \n, \theta \rangle}, \ M_\theta f = \sum_j
\gamma_\j(\theta) \, \chi_j.
\end{eqnarray*}

Hence, $M_\theta f$ is defined for every $\theta$, if $\sum_{j \in
J}\big(\sum_{\n \in \Z^d} |c_f(T^{\tn}\, \chi_j)|^2) < \infty$.

Let us assume that $f$ has an absolutely convergent Fourier series.
Since every $\chi \in \tilde H^*$ can be written in an unique way as
$\chi = T^{\tn}\, \chi_j$, with $j \in J$ and $\n \in \Z^d$, we have
$$\sum_{\tn \in \Z^d} |c_f(T^{\tn}\, \chi_j)| \leq  \sum_{j \in J}
\sum_{\tn \in \Z^d} |c_f(T^{\tn}\, \chi_j)| =  \sum_{\chi \in \tilde
H^*} |c_f(\chi)| = \|f \|_{c}.$$ Then, for every $j \in J$, the
series defining $\gamma_j$ is uniformly converging, $\gamma_j$ is
continuous and $\sum_{j \in J} \|\gamma_j\|_\infty \leq \sum_{j \in
J} \sum_{\tn \in \Z^d} |c_f(T^{\tn}\, j)| = \|f \|_{c}$. The
function $\sum_{j \in J} |\gamma_j|^2$ is a continuous version of
the spectral density $\varphi_f$. Therefore Lemma \ref{approxMart}b
applies.

\vskip 3mm The following lemma  allows to check the condition of
Lemma \ref{approxMart}a in terms of the Fourier series of $f$.

\begin{lem} \label{lpSpec} If $\sum _{\chi \in \hat G ^*} |c_f(\chi)|^r < +\infty$ for some $1< r \leq 2$,
then the spectral density $\varphi_f$, if $\mu(f) =0$, is in
$L^p(\T^d)$ with $p = \frac12{r \over r-1} > 1$.
\end{lem} \begin{proof} We have $\varphi_f(t) = \sum_{j \in J} |\gamma_j(t)|^2$, with
$\gamma_j(t) = \sum_{\n \in \Z^d \stm0} c_f(A^\n \chi_j) e^{2\pi i
\langle \n, t \rangle}$.

Let $p := \frac12{r \over r-1}$ be such that $r$ and $2p$ are conjugate exponents.
By the Hausdorff-Young theorem, we have:
$$\|\gamma_j\|_{2p} \leq (\sum _{\n \in \Z^d \stm0} |c_f(A^\n \chi_j)|^r)^{1/r}.$$
It follows, since $2/r \geq 1$:
\begin{eqnarray*}
\sum_{j \in J} \| |\gamma_j|^2\|_p &&= \sum_{j \in J} \| \gamma_j\|_{2p}^2
\leq \sum_{j \in J} (\sum _{\n \in \Z^d \stm0} |c_f(A^\n \chi_j)|^r)^{2/r} \\
&& \leq (\sum_{j \in J} \sum _{\n \in \Z^d \stm0} |c_f(A^\n \chi_j)|^r)^{2/r}
\leq (\sum _{\chi \in \hat G ^*} |c_f(\chi)|^r)^{2/r}. \label{lpnormIn}
\end{eqnarray*}
Therefore we have $\varphi_f \in L^p(\T^d)$ and $\|\varphi_f\|_p
\leq (\sum _{\chi \in \hat G ^*} |c_f(\chi)|^r)^{2/r}$.
\end{proof}
Remark that the condition of the lemma for $r=1$ corresponds to $f
\in AC_0(G)$, which implies continuity of $\varphi_f$.

\goodbreak
\vskip 5mm \subsection{\bf Mixing, moments and cumulants,
application to the CLT}

\

{\bf Reminders on moments and cumulants}

Before we continue studying actions by automorphisms, for the sake
of completeness, we recall in this subsection some general results
on mixing of all orders, moments and cumulants (see \cite{GneKol54},
\cite{Leo60c} and the references given therein). Implicitly we
assume existence of moments of all orders when they are used.

For a real random variable $Y$ (or for a probability distribution on
$\R$), the cumulants (or semi-invariants) can be formally defined as
the coefficients $c^{(r)}(Y)$ of the cumulant generating function $t
\to \ln \E(e^{tY}) = \sum_{r=0}^\infty  \, c^{(r)}(Y) \,{t^r \over
r!}$, i.e.,
\begin{eqnarray*}
c^{(r)}(Y) = {\partial^r \over \partial^r t} \ln \E(e^{tY})|_{t =
0}.
\end{eqnarray*}
Similarly the joint cumulant of a random vector $(X_1, ... , X_r)$
is defined by
\begin{eqnarray*}
c(X_1, ... , X_r) = {\partial^r \over \partial t_1 ... \partial t_r}
\ln \E(e^{\sum_{j=1}^r t_jX_j})|_{t_1 = ... = \, t_r = 0}.
\end{eqnarray*}
This definition can be given as well for a finite measure on $\R^r$.

One easily checks that the joint cumulant of $(Y, ...,Y)$ ($r$
copies of $Y$) is $c^{(r)}(Y)$.

 For any subset $I = \{i_1, ..., i_p\} \subset J_r:= \{1, ..., r\}$,
we put
$$m(I) = m(i_1, ..., i_p):= \E(X_{i_1} ... X_{i_p}),
\  s(I)= s(i_1, ..., i_p):= c(X_{i_1}, ... , X_{i_p}).$$

The cumulants of a process $(X_j)_{j \in \Cal J}$, where $\Cal J$ is
a set of indexes, is the family
$$\{c(X_{i_1}, ..., X_{i_r}), ({i_1}, ..., {i_r}) \in \Cal J^r, r \geq 1\}.$$

The following formulas link moments and cumulants and vice-versa:
\begin{eqnarray}
c(X_1, ... , X_r) &=& s(J_r) = \sum_{\Cal P} (-1)^{p-1} (p - 1)!
\ m(I_1) ...  m(I_p), \label{cumFormu1}\\
\E(X_{1} ... X_{r}) &=& m(J_r) = \sum_{\Cal P} s(I_1) ... s(I_p).
\label{cumFormu2}
\end{eqnarray}
where in both formulas, $\Cal P = \{I_1, I_2, ..., I_p\}$ runs
through the set of partitions of $J_r = \{1, ..., r\}$ into $p \le
r$ non empty intervals.

\vskip 3mm Now, let be given a random process $(X_\k)_{\k \in
\Z^d}$, where for $\k \in \Z^d$, $X_\k$ is a real random variable,
and a summation kernel $R$ with finite support in $\Z^d$ and values
in $\R^+$. (For examples of summation kernels, see Section
\ref{specAnaly}, in particular Proposition \ref{barycenter}). Let us
consider the process defined for $\k \in \Z^d$ by
$$Y_k = \sum_{\el \in \Z^d} \, R(\el+\k) \, X_\el, \,\k \in \Z^d.$$

By permuting summation and integral, we easily obtain:
\begin{eqnarray*}
c(Y_{\k_1}, ... , Y_{\k_r}) =  \sum_{(\el_1, ..., \el_r) \, \in
(\Z^d)^r} \, c(X_{\el_1}, ... , X_{\el_r}) \,  R(\el_1 +\k_1) ...
R(\el_r + \k_r).
\end{eqnarray*}

In particular, we have for $Y = \sum_{\el \in \Z^d} \, R(\el) \,
X_\el$:
\begin{eqnarray}
c^{(r)}(Y) = c(Y, ..., Y) =  \sum_{(\el_1, ..., \el_r) \, \in
(\Z^d)^r} \, c(X_{\el_1}, ... , X_{\el_r}) \,  R(\el_1) ...
R(\el_r). \label{cumLin0}
\end{eqnarray}

\vskip 3mm {\it Limiting distribution and cumulants}

For our purpose, we state in terms of cumulants a particular case of
a theorem of M. Fr\'echet and J. Shohat, generalizing classical
results of A. Markov. Using the formulas linking moments and
cumulants, a special case of the ``generalized statement of the second
limit-theorem" given in \cite{FreSho31} can be expressed follows:

\begin{thm} \label{FreSho} Let $(Z^n, n \geq 1)$ be a sequence
of centered r.r.v. such that
\begin{eqnarray}
\lim_n c^{(2)}(Z^n) = \sigma^2, \ \lim_n c^{(r)}(Z^n) = 0, \forall r
\geq 3, \label{limCum}
\end{eqnarray}
then $(Z^n)$ tends in distribution to $\Cal N(0, \sigma^2)$. (If
$\sigma = 0$, then the limit is $\delta_0$).
\end{thm}

It implies the following result (cf. Theorem 7 in \cite{Leo60b}):
\begin{thm} \label{Leonv} Let $(X_\k)_{\k\in \Z^d}$ be a random process and
$(R_n)_{n \geq 1}$ a summation sequence on $\Z^d$. Let $(Y^n)_{n
\geq 1}$ be the process defined by $Y^n = \sum_\el \, R_n(\el) \,
X_\el, n \geq 1$. Under the assumptions $\lim_n \|Y^n\|_2 = +\infty$
and
\begin{eqnarray}
\sum_{(\el_1, ..., \el_r) \, \in (\Z^d)^r} \, c(X_{\el_1}, ... ,
X_{\el_r}) \,  R_n(\el_1) ... R_n(\el_r)  = o(\|Y^n\|_2^r), \forall
r \geq 3, \label{smallCumul}
\end{eqnarray}
${Y^n \over \|Y^n\|_2}$ tends in distribution to $\Cal N(0, 1)$ when
$n$ tends to $\infty$.
\end{thm}
\begin{proof} Let $\beta_n := \|Y^n\|_2 = \|\sum_\el \, R_n(\el) \,
X_\el\|_2$ and $Z_n = \beta_n^{-1} Y^n$.

We have using (\ref{cumLin0}), $c^{(r)}(Z^n)  = \beta_n^{-r}
\sum_{(\el_1, ..., \el_r) \, \in (\Z^d)^r} \, c(X_{\el_1}, ... ,
X_{\el_r}) \, R(\el_1) ... R(\el_r)$. The theorem follows then from
the assumption (\ref{smallCumul}) by Theorem \ref{FreSho} applied to
$(Z_n)$.
\end{proof}

\begin{defn} A measure preserving $\N^d$- or $\Z^d$-action $T: {\tn} \to T^\tn$
on a probability space $(\Omega, \Cal A, \mu)$ is $r$-mixing, $r >
1$, if for all sets $B_1, ..., B_r \in \Cal A$
$$\lim_{\min_{{1 \leq \ell < \ell' \leq r}} \|\n_\ell - \n_{\ell'}\| \to
\infty} \mu(\bigcap_{\ell = 1}^r T^{-\n_\ell} \, B_\ell)=\prod_{\ell
= 1}^r \mu(B_\ell).$$
\end{defn}

\begin{nota} {\rm For $f \in L_0^\infty$, the space of measurable essentially bounded
functions on $(\Omega, \mu)$ with $\int f \, d\mu = 0$, we apply the
definition of moments and cumulants to $(T^{\n_1}f,..., T^{\n_r}f)$
and put}
\begin{eqnarray}
m_f(\n_1,..., \n_r) = \int \, T^{\n_1} f... T^{\n_r} f \, d\mu, \ \
s_f(\n_1,..., \n_r) := c(T^{\n_1} f,..., T^{\n_r} f). \label{notasf}
\end{eqnarray}
\end{nota}

In order to show that the cumulants of a system which is mixing of all orders
are asymptotically null, we need the following lemma.

\begin{lem} \label{subSeqLem} For every sequence $(\n_1^k,..., \n_r^k)$ in
$(\Z^d)^r$, there are a subsequence with possibly a permutation of
indices (still written $(\n_1^k,..., \n_r^k)$), an integer
$\kappa(r) \in [1, r]$ and a subdivision $1 = r_{1} < r_{2} < ... <
r_{{\kappa(r)-1}} <  r_{{\kappa(r)}} \leq r$ of $\{1, ..., r\}$,
such that
\begin{eqnarray}
&&\lim_k \min_{1 \leq s \not = s' \leq \kappa(r)}\|\n_{r_{s}}^k -
\n_{r_{s'}}^k\| = \infty, \label{minCond}\\
&&\n_j^k = \n_{r_{s}}^k + \a_j, \text{ for } r_{s} < j < r_{{s+1}},
\ s = 1, ..., \kappa(r)-1, \text{ and for } r_{\kappa(r)} < j \leq
r, \label{constant1}
\end{eqnarray}
where $\a_j$ is a constant integral vector.

If the sequence $(\n_1^k,..., \n_r^k)$ satisfy $\lim_k \max_{i \not
= j} \|\n_i^k - \n_j^k\| = \infty$, then the construction can be
done in such a way that $\kappa(r) > 1$.
\end{lem}

Remark that if $\sup_k \max_{i \not = j} \|\n_i^k - \n_j^k\| <
\infty$, then $\kappa(r) = 1$ so that (\ref{minCond}) is void and
(\ref{constant1}) is void for the indexes such that $r_{{s+1}} =
r_{s} +1$.
\begin{proof} The proof is by induction. The result is clear for $r=
2$. Suppose we have construct the subsequence for the sequence of
$r-1$-tuples $(\n_1^k,..., \n_{r-1}^k)$.

Let $1 \leq r_{1} < r_{2} < ... < r_{{\kappa(r-1)}} \leq r-1$ be the
corresponding subdivision of $\{1, ..., r -1\}$, as stated above for
the sequence $(\n_1^k,..., \n_{r-1}^k)$. If the sequence
$(\n_1^k,..., \n_{r-1}^k)$ satisfy $\lim_k \max_{1 < i < j \leq r-1}
\|\n_i^k - \n_j^k\| = \infty$, then $\kappa(r-1) > 1$ by
construction in the induction process.

Now we consider $(\n_1^k,..., \n_r^k)$. If $\lim_k \|\n_r^k -
\n_i^k\| = +\infty$, for all $i=1, ..., r-1$, then we have just to
take $1 \leq r_{1} < r_{2} < ... < r_{{\kappa(r-1)}} <
r_{{\kappa(r)}} = r$ as new subdivision of $\{1, ..., r\}$.

If $\liminf_k \|\n_r^k - \n_{i_s}^k\| < +\infty$, for some $s \leq
\kappa(r-1)$, then along a new subsequence (still denoted with the
same notation) we have $\n_r^k = \n_{i_s}^k + \a_r$, where $\a_r$ is
a constant integral vector. After changing the labels, we insert
$n_r$ in the subdivision for $\{1, ..., r-1\}$ and obtain the new
subdivision for $\{1, ..., r\}$.

For the last condition on $\kappa$, suppose that $\lim_k \max_{1 < i
< j \leq r} \|\n_i^k - \n_j^k\| = \infty$.

Then if $\liminf_k \max_{1 < i < j \leq r-1} \|\n_i^k - \n_j^k\| <
+\infty$, necessarily, $\kappa(r) > 1$. If, on the contrary, the
sequence $(\n_1^k,..., \n_{r-1}^k)$ satisfy $\lim_k \max_{1 < i < j
\leq r-1} \|\n_i^k - \n_j^k\| = \infty$, then $\kappa(r-1) > 1$ so
that $\kappa(r) \geq \kappa(r-1) > 1$.
\end{proof}

\begin{lem} \label{skTozeroLem} If a $\Z^d$-dynamical system is mixing
of order $r \geq 2$, then, for any $f \in L_0^\infty$,
\begin{eqnarray}
\underset{\max_{i \not = j} \|\n_i - \n_j\| \to \infty} \lim
s_f(\n_1,..., \n_r) = 0. \label{skTozeroForm}
\end{eqnarray}
\end{lem}

\begin{proof} We give a sketch of the proof. The notation $s_f$
was introduced in (\ref{notasf}). Suppose that (\ref{skTozeroForm})
does not hold. Then there is $\varepsilon
> 0$ and a sequence of $r$-tuples $(\n_1^k =\0,..., \n_r^k)$ such
that $|s_f(\n_1^k,..., \n_r^k)| \geq \varepsilon$ and $\max_{i \not
= j} \|\n_i^k - \n_j^k\| \to \infty$ (we use stationarity).

By taking a subsequence (but keeping the same notation), we can
assume that, for two fixed indexes $i, j$, $\lim_k \|\n_i^k -
\n_j^k\| = \infty$.

From Lemma \ref{subSeqLem}, it follows that there is a subdivision
$1 = r_{1} < r_{2} < ... < r_{{\kappa(r)-1}} <  r_{{\kappa(r)}} \leq
r$ and constant integer vectors $\a_j$ such that
\begin{eqnarray}
&&\lim_k \min_{1 \leq s \not = s' \leq \kappa(r)}\|\n_{r_{s}}^k -
\n_{r_{s'}}^k\| = \infty, \label{minCond2}\\
&&\n_j^k = \n_{r_{s}}^k + \a_j, \text{ for } r_{s} < j < r_{{s+1}},
\ s = 1, ..., \kappa(r)-1, \text{ and for } r_{\kappa(r)} < j \leq
r.
\end{eqnarray}

Let $d\mu_k(x_1, ..., x_r)$ denote the probability measure on $\R^r$
defined by the distribution of the random vector $(T^{\n_1^k} f(.),
..., T^{\n_r^k}f(.))$. We can extract a converging subsequence from
the sequence $(\mu_k)$, as well as for the moments of order $\leq
r$.

Let us denote $\nu(x_1, ..., x_r)$ (resp. $\nu(x_{i_1}, ...,
x_{i_p})$) the limit of $\mu_k(x_1, ..., x_r)$ (resp. of its
marginal measures $\mu_k(x_{i_1}, ..., x_{i_p})$ for $\{i_1, ...,
i_p\} \subset \{1, ..., r\}$).

Let $\varphi_i, i=1, ..., r$, be continuous functions with compact
support on $\R$. Mixing of order $r$ and condition (\ref{minCond2})
imply
\begin{eqnarray*}
&&\nu(\varphi_1 \otimes \varphi_2 \otimes  ...  \otimes \varphi_r)
=\lim_k \int_{\R^d} \varphi_1 \otimes \varphi_2 \otimes  ... \otimes
\varphi_r \, d\mu_k =
\lim_k \int \prod_{i=1}^r \varphi_i(f(T^{\n_i^k}x))\  d\mu(x) \\
&&=\lim_k \int \ \bigl[\prod_{s=1}^{\kappa(r)-1} \ \prod_{r_{s} \leq
j < r_{{s+1}}} \varphi_j(f(T^{\n_s^k + \a_j} x))\bigr] \,
\prod_{\kappa(r) \leq j \leq r} \varphi_j(f(T^{\n_{\kappa(r)}^k + \a_j} x)) \ d\mu(x) \\
&&= \bigl[\prod_{s=1}^{\kappa(r)-1} \, \bigl(\int \prod_{r_{s} \leq
j < r_{{s+1}}} \varphi_j(f(T^{\a_j}x)) \ d\mu(x) \bigr)\bigr] \,
\bigl[\int \prod_{\kappa(r) \leq j \leq r} \varphi_j(f(T^{\a_j} x))
\ d\mu(x)\bigr].
\end{eqnarray*}

Therefore $\nu$ is the product of marginal measures corresponding to
disjoint subsets: at least there are $I_1 = \{i_1, ..., i_p\}, I_2 =
\{i_1', ..., i_{p'}'\} \subset J_r = \{1, ..., r\}$, two non empty
subsets such that $(I_1, I_2)$ is a partition of $J_r$ and
$d\nu(x_1, ...,x_r) = d\nu(x_{i_1}, ..., x_{i_p}) \times
d\nu(x_{i_1'}, ..., x_{i_{p'}'})$.

 Putting $\Phi(t_1, ..., t_r) =
\ln \int e^{\sum t_j x_j} \ d\mu(x_1, ...,x_r)$ and the analogous
formulas for $\nu(x_{i_1}, ..., x_{i_p})$ and $\nu(x_{i_1'}, ...,
x_{i_p'})$, we obtain: $\Phi(t_1, ..., t_r) = \Phi(t_{i_1}, ...,
t_{i_p}) + \Phi(t_{i_1'}, ..., t_{i_p'})$. It implies that the
derivative $ {\partial^r \over \partial t_1 ...
\partial t_r}\Phi(t_1, ..., t_r)|_{t_1 = ... = \, t_r = 0}$ is 0.
Hence $c(\nu(x_1, ..., x_r)) = 0$.

But this contradicts $\liminf_k |s_f(\n_1^k,..., \n_r^k)| > 0$.
\end{proof}

\vskip 3mm {\bf Application to $d$-dimensional actions by
endomorphisms}

 For an action of $\N^d$ by commuting endomorphisms on $(G, \mu)$,
a compact abelian group with Haar measure $\mu$, the method of moments
as in \cite{Leo60b} can be used for the CLT when mixing of all
orders is satisfied. It gives immediately the CLT for trigonometric
polynomials.

\begin{prop} \label{tclPol0} Let $\n: (n_1, ..., n_d) \to T^\n =
T_1^{n_1}... T_d^{n_d}$ be a totally ergodic $\N^d$-action by commuting
endomorphisms on a compact abelian group $G$ which is mixing of all orders.
Let $(R_n)_{n \geq 1}$ be a summation sequence
on $\N^d$ and let $f$ be a trigonometric polynomial. If $\lim_n \|\sum_\ell
R_n(\el) \, T^\el f \|_2 = \infty$, then the CLT is satisfied by the
sequence $({\sum_\el R_n(\el) \, T^\el f \over \|\sum_\el R_n(\el) \, T^\el f \|_2})_{n \geq 1}$.
\end{prop}
\begin{proof} For an action by endomorphisms of compact
abelian groups, the moments of the process $(f(T^\n .))_{\n \in
\Z^d}$ for a trigonometric polynomial $f(x) = \sum_{\k \in \Lambda}
c_\k(f) \, \chi_\k(x)$ are:
\begin{eqnarray*}
&&m_f(\n_1, ..., \n_r) = \int f(T^{\n_1} x) ... f(T^{\n_r} x) \ dx =
\sum_{\k_1, ..., \k_r \in \Lambda} c_{k_1} ... c_{k_r} 1_{T^{n_1}
\chi_{\k_1} ... T^{\n_r} \chi_{\k_r} = \, {\bf 1}}.
\end{eqnarray*}
 For $r$ fixed, the function $(\k_1, ..., \k_r) \to m_f(\k_1, ...,
\k_r)$ takes a finite number of values, since by the above formula
$m_f$ is a sum with coefficients 0 or 1 of the products $c_{k_1} ...
c_{k_r}$ with $k_j$ in a finite set. The cumulants of a given order
according to (\ref{cumFormu1}) take also a finite number of values.

Therefore, since mixing of all orders implies
$\underset{\max_{i,j} \|\el _i - \el _j\| \to
\infty} \lim \, s_f(\el _1,..., \el _r) = 0$ by Lemma \ref{skTozeroLem},
there is $M_r$ such that $s_f(\el _1, ..., \el _r) = 0$ for
$\max_{i,j} \|\el _i - \el_j\| > M_r$.

We apply Theorem \ref{Leonv}. Let us check (\ref{smallCumul}). Using
(\ref{cumLin0}), we obtain that
\begin{eqnarray*}
|\sum s_f(\el _1, ..., \el _r) \, R_n(\el _1)...R_n(\el _r)| &&=
|\sum s_f(\0, \el _2-\el _1, ..., \el _r -\el _1) \, R_n(\el _1)...R_n(\el _r)|\\
&&\leq \sum_{\|\el _2'\|, ..., \|\el _r'\| \leq M_r} |s_f(\0, \el
_2', ..., \el _r')| \, \sup_n \|R_n\|_\infty^r.
\end{eqnarray*}
Since the summation sequences are supposed to be bounded, $\sum
s_f(\el _1, ..., \el _r) \, R_n(\el _1)...R_n(\el _r)$ is bounded
and (\ref{smallCumul}) is satisfied.
\end{proof}

Remark that if a subsequence of $\lim_n \|\sum_\ell R_n(\el) \,
T^\el f \|_2$ is bounded, then $f$ is a double coboundary. Indeed,
supposing for simplicity $d=2$, we have the following lemma:
\begin{lem} Let $T_1$ and $T_2$ generate a mixing (2-mixing) $\mathbb N^2$-action.
For $f \in L_0^2(\mu)$, the following statements are equivalent:
\hfill \break (i) $f=(I-T_1)(I-T_2)g$, for some $g\in L_0^2(\mu)$,
\hfill \break (ii)
$\liminf_{m}\|\sum_{j=0}^{m-1}\sum_{k=0}^{m-1}T_1^jT_2^kf\|
<\infty$.
\end{lem}
\begin{proof}
Let us prove $(ii)\Rightarrow(i)$. Let $(m_\ell)$ be an increasing
subsequence for which $\|\sum_{j, \, k=0}^{m_\ell-1}T_1^j
T_2^kf\|<\infty$. We may take a subsequence of $(m_\ell)$, still
denoted by $(m_\ell)$, for which $\sum_{j, \, k=0}^{m_\ell-1}T_1^j
T_2^k f\underset{\ell\to\infty}{\longrightarrow} g$ weakly. Hence,
the following weak convergence holds
$$(I-T_1)(I-T_2)g=\lim_\ell (I-T_1^{m_\ell})(I-T_2^{m_\ell})f=f-\lim_\ell
(T_1^{m_\ell}f+T_2^{m_\ell}f-T_1^{m_\ell}T_2^{m_\ell}f)=f$$
\end{proof}

\vskip 3mm \subsection{\bf CLT for compact abelian connected groups}

\

There is a subclass of actions by automorphisms satisfying the
$K$-property and this is a way to prove the CLT in that case,
using martingale-type property, as shown in \cite{CohCo12}. Let us
mention that, for $\Z^d$-action by automorphisms on zero-dimensional
compact abelian groups, the $K$-property (or property of completely
positive entropy) is equivalent to mixing of all orders (cf.
\cite{SchWar93}). We will rather focus here on an extension of the
method of $r$-mixing used by Leonov for a single ergodic
automorphism and use it for abelian groups of toral automorphisms.

\vskip 2mm {\bf The method of Leonov}

The proof of the CLT given by Leonov in \cite{Leo60b} for a single
ergodic automorphism $T$ of a compact abelian group $G$ is based on
the computation of the moments, when $f$ is trigonometric
polynomial. It uses the fact that $T$ is mixing of all orders, a
property shown by Rohlin \cite{Roh61}, consequence of the
$K$-property for ergodic automorphisms.

For $\Z^d$-actions by automorphisms on connected compact abelian
groups, in particular on tori, the method of moments can also be
used, since the mixing property of all orders holds (Theorem
\ref{r-mixing} below). First we prove a CLT for trigonometric
polynomials using the mixing property, then the result is extended
to regular functions by approximation.

\vskip 3mm {\bf Mixing of actions by endomorphisms ($G$ connected)}
\label{mixAuto}

For $\Z^d$-actions by automorphisms on compact abelian groups
mixing of all orders is not always satisfied (cf. \cite{Le78}, \cite{Sc95}).
In 1992, W. Philip \cite{Ph94} and K. Schmidt and T. Ward
\cite{SchWar93} used results about the number of solutions of
$S$-units equations in the study of semigroups of endomorphisms or
automorphisms on compact abelian groups. In particular, the following
mixing of all orders for $\Z^d$-actions by automorphisms on connected groups
is a consequence of algebraic results on $S$-units (\cite{Schl90}):

\begin{thm} \label{r-mixing} (\cite [Corollary 3.3] {SchWar93})
Let $ \n \to T^\n$ be a mixing $\Z^d$-action on a compact,
connected, abelian group G. Then it is $r$-mixing for every $r \geq
2$.\end{thm}

\begin{cor} \label{r-mixEndo} Let ${\Cal S}$ be a semigroup of
endomorphisms on a compact connected abelian group $G$. If this action is totally
ergodic, it is mixing of all orders.
\end{cor}
\begin{proof} \ We use the notations of Lemma \ref{embedd}.
The group $G$ in Lemma \ref{embedd} is connected and Theorem
\ref{r-mixing} applies to the group $\tilde {\Cal S}$ of
automorphisms of $G$ in which ${\Cal S}$ is embedded. The action of
$\tilde {\Cal S}$ is mixing of all orders, hence also the action of
${\Cal S}$.
\end{proof}

\vskip 3mm {\bf CLT for $\N^d$-action by endomorphisms}

\begin{thm} \label{tclPol0A} Let $\n: (n_1, ..., n_d) \to T^\n =
T_1^{n_1}... T_d^{n_d}$ be a totally ergodic $\N^d$-action by
commuting endomorphisms on a compact abelian connected group $G$.
Let $(D_n)_{n \geq 1}$ be a F\o{}lner sequence in $\N^d$ and let $f$
be in $AC_0(G)$. We have\footnote{with the convention that the
limiting distribution is $\delta_0$ if $\sigma^2(f) = 0$.}
$\sigma^2(f) = \varphi_f(0)$ and
$$ (\sum_{\el \in \N^d} R_n(\el)^2)^{-\frac12} \, \sum_{\el \in \N^d} R_n(\el) f(A^\el \, .)
 \overset{distr} {\underset{n \to
\infty} \longrightarrow } \Cal N(0,\sigma^2(f)).$$
\end{thm}
\begin{proof} Let $(\Cal N_s)$ be an increasing
sequence of finite sets in $\hat G$ with union $\hat G \setminus \{0\}$ and
let $f_s(x) := \sum_{\chi \in \Cal N_s} \, c_f(\chi) \, \chi$ be the
trigonometric polynomial obtained by restriction of the Fourier
series of $f$ to $\Cal N_s$. Let $Z_n^{s}$, $Z_n$ denote
respectively
\begin{eqnarray*}
Z_n^{s} := |D_n|^{-\frac12} \, \sum_\el R_n(\el) f_s(A^\el .), \
Z_n := |D_n|^{-\frac12} \, \sum_\el R_n(\el) f(A^\el .).
\end{eqnarray*}
By Theorem \ref{condAbsConv} we have $\sigma(f - f_s) \leq
\|f-f_s\|_2$. It follows $\sigma^2(f):=\lim_s \sigma^2(f_s)$ and
$\sigma^2(f_s) \not = 0$ for $s$ big enough, since $\sigma^2(f) > 0$
by hypothesis.

 For the kernel $K_n$ associated to $D_n$, by the F\o{}lner property
and (\ref{ConvVarRotergSumDn}), we have, if $g$ satisfies
(\ref{absConv}):
$$|D_n|^{-1} \, \|\sum_\el R_n(\el) \, A^\el g \|_2^2 =
\int_{\T^d} \, K_n \, \varphi_g \, dt \underset{n \to \infty} \to
\varphi_g (0) = \sigma^2(g).$$

 From Theorem \ref{r-mixing} (mixing of all orders) and the result for
the trigonometric polynomial $f_{s}$, it follows:
${Z_n^{s} \overset{distr} {\underset{n \to \infty} \longrightarrow}
\Cal N(0,\sigma^2(f_s))}$ for every $s$. Moreover, since
\begin{eqnarray*}
\limsup_n \int |Z_{n}^s - Z_n|_2^2 \ d\mu &&=
\limsup_n \int_{\T^d} \, K_n \,\varphi_{f-f_s} \, dt \\
&&= \lim_n \int_{\T^d} \, K_n \,\varphi_{f-f_s} \, dt =
\sigma^2(f-f_s) \leq C \|f - f_s\|_2^2,
\end{eqnarray*}
we have $\limsup_n \mu[|Z_{n}^s - Z_n| > \varepsilon] \leq \varepsilon^{-2}
\limsup_n \int |Z_{n}^s - Z_n|_2^2 \ d\PP \underset{s \to 0}  \to 0$
for every $\varepsilon > 0$ and the condition $\lim_s \limsup_n \PP[|Z_{n}^s - Z_n|
> \varepsilon] = 0$ is satisfied.

The conclusion $Z_n \overset{distr} {\underset{n \to \infty} \longrightarrow } \Cal
N(0,\sigma^2(f))$ follows from Theorem 3.2 in \cite{Bill99},
\end{proof}

We have in particular:
$$|D_n|^{-\frac12} \, \sum_{\el \in D_n} f(A^\el \, .)
\overset{distr} {\underset{n \to \infty} \longrightarrow } \Cal
N(0,\sigma^2(f)).$$

The previous result is valid for the rotated sums: if $f$ in
$AC_0(G)$, then, for every $\theta$,
\begin{eqnarray}
\sigma_\theta^2(f) = \varphi_f(\theta), \ |D_n|^{-\frac12} \,
\sum_{\el \in D_n} e^{2 \pi i \langle\el, \theta\rangle} f(A^\el .)
\overset{distr} {\underset{n \to \infty} \longrightarrow } \Cal
N(0,\sigma_\theta^2(f)). \label{CLTRot1}
\end{eqnarray}
If $f$ satisfies the regularity condition (\ref{regFour2b2}), then
$\sigma_\theta^2(f) = 0$ if and only if there are continuous
functions $u_{t,\theta}$ on $\T^\rho$, for $t= 1, ..., d$, such that
$f= \sum_{t=1}^d (I- e^{2\pi i \theta} A_t) u_{t,\theta}$.

This applies in particular when $(D_n)$ is a sequence of
$d$-dimensional cubes in $\Z^d$.

\vskip 3mm {\it A CLT for the rotated sums for a.e. $\theta$ without
regularity assumptions}

 For the summation sequence given by $d$-dimensional cubes, a CLT for
the rotated sums can be shown for a.e. $\theta$ without regularity
assumptions on $f$. The proof relies on (\ref{approxtheta2}) which
is satisfied, for any given $f \in L^2(G)$, for $\theta$ in a set of
full measure. This extends results of \cite{CohCo12}.
\begin{thm} \label{rotatedCLT} Let $\n \to A^\n$ be a totally ergodic
$d$-dimensional action by commuting endomorphisms on $G$. Let
$(D_n)_{n \geq 1}$ be a sequence of cubes in $\Z^d$. Let $f \in
L^2(G)$. For a.e. $\theta \in \T^d$, we have $\sigma_\theta^2(f) =
\varphi_f(\theta)$ and
$$|D_n|^{-\frac12} \, \sum_{\el \in D_n} e^{2 \pi i \langle\el, \theta\rangle}
f(A^\el .) \overset{distr} {\underset{n \to \infty} \longrightarrow
} \Cal N(0,\sigma_\theta^2(f)).$$
\end{thm}

Let us mention that, if we take for $D_n$ triangles instead of
squares, a CLT for the rotated sums is also valid for a.e. $\theta$,
provided $f$ satisfies $\sum_{\chi \in \hat G} |c_f(\chi)|^r <
+\infty$, for some $r < 2$.

\vskip 3mm \goodbreak {\it Other examples of kernels}

\begin{thm} \label{tclRegKer} Let $(R_n)_{n \geq 1}$ be a summation
sequence on $\Z^d$ which is regular and such that $(\tilde R_n(t))$
weakly converges to a measure $\zeta$ on the circle. Let $f$ be a
function in $AC_0(G)$ with spectral density $\varphi_f$. If
$\zeta(\varphi_f) \not = 0$, then we have
$$\sum_{\el \in \Z^d} R_n(\el ) f(A^\el .)  /
(\sum_{\el \in \Z^d} \, |R_n(\el)|^2)^\frac12 \overset{distr}
{\underset{n \to \infty} \longrightarrow } \Cal N(0,
\zeta(\varphi_f)).$$
\end{thm}
\begin{proof}
The proof is the same as that of Theorem \ref{tclPol0A} and uses
(\ref{densspectAC}) and the convergence:
\begin{eqnarray*}
\lim_n  \|\sum_{\el \in \Z^d} \, R_n(\el ) \, A^{\el}f\|_2^2 /
\sum_{\el \in \Z^d} \, |R_n(\el)|^2 = \lim_n \int_{\T^d} \, K_n(t)
\, \varphi_f(t) \, dt = \zeta(\varphi_f).
\end{eqnarray*}
\end{proof}

\vskip 3mm {\bf Barycenter operators}

The iterates of the barycenter operators satisfy the condition of
Theorem \ref{tclRegKer}.

Let $A_1, ...,  A_d$ be to commuting endomorphisms of a connected
abelian group $G$ generating a totally ergodic action. Let $P$ be
the barycenter operator defined as in Formula (\ref{defBary}) by:
\begin{eqnarray} Pf(x) := \sum_j \, p_j \, f(A_j x).\label{defBary2}
\end{eqnarray}

Observe that the coefficient $R_n(\el)$ of the expansion of the
summation sequence associated to $n^{d-1 \over 4}P^n$ tends to 0
uniformly, when $n$ tends to infinity (to prove it, one can use the
local limit theorem for multinomial Bernoulli variables).

By Proposition \ref{barycenter} and Theorem \ref{tclRegKer} we
obtain:
\begin{thm} \label{tclBaryc} Let $f$ be a function in $AC_0(G)$ with spectral
density $\varphi_f$. Assume that $\sigma_P^2(f) :=  \int_\T \,
\varphi_{f}(u,u, ...,u) \, du \not = 0$. Then we have
$$(4\pi)^{{d-1 \over 4}} \, (p_1 ... p_{d})^{\frac14} \, n^{d-1 \over 4}
\, P^n f(.) \overset{distr} {\underset{n \to \infty} \longrightarrow
} \Cal N(0, \sigma_P^2(f)).$$
\end{thm}

For the torus, the result holds for $f$ satisfying Condition
(\ref{5.11}) in the next subsection.

 {\it Example:} let $A_1,
A_2$ be two commuting matrices in ${\Cal M}^*(\rho, \Z)$ generating
a totally ergodic action on $\T^\rho$, $\rho \geq 3$. Let $P$ be the
barycenter operator: $Pf(x) := \frac12(f(A_1 x) + f(A_2 x))$.

If $\varphi_f$ is continuous, then we have $\lim_{n\to\infty}
\sqrt{\pi n} \, \|P^n f\|_2^2 = \int_\T \, \varphi_{f}(u,u) \, du$.
It follows from Theorem \ref{tclBaryc}, for $f$ satisfying
(\ref{5.11}) on $\T^\rho$:
$$(\pi n)^\frac14 \,P^n f \overset{distr} {\underset{n \to \infty} \longrightarrow }
\Cal N(0,\sigma_P^2(f)).$$

For any non trivial character $\chi$ on $G$, the spectral density is
identically 1 and $\sigma_P(\chi) = 1$. The rate of convergence of
$\|P^n \chi\|_2$ to 0 is the polynomial rate given by Proposition
\ref{barycenter}. If $f$ is in $AC_0(G)$, we have $\sigma_P(f) = 0$
if and only if $\varphi_{f}(u,u) = 0$, for every $u \in \T^1$. In
particular, by the results of Subsection \ref{endoTori}, if $f$ is
not a mixed coboundary (cf. (\ref{cobound10A})), then $\sigma_P(f)
\not = 0$ and the rate of convergence of $\|P^nf\|_2$ to 0 is the
polynomial rate given by Proposition \ref{barycenter}. A test of non
degeneracy on periodic points can be deduced from it.

The condition $\sigma_P(f) = 0$ is stronger than the coboundary
condition. A sufficient condition to have $\sigma_P(f) = 0$  is that
$f$ can be written $f = A_1 g - A_2 g$ with $g \in L_0^2(\mu)$.

\vskip 3mm {\bf Remarks.} 1) The case of commutative or amenable
actions strongly differs from the case of non amenable actions for
which a ``spectral gap property" is often available (\cite{FuSh99}).
For action by algebraic (non commuting) automorphisms $A_j, j=1,
..., d$, on the torus, the existence of a spectral gap for $P$ of
the form (\ref{defBary2}) is related to the fact that the generated
group has no factor torus on which it is virtually abelian
(\cite{BeGu11}).

\vskip 2mm 2) For $\nu$ a discrete measure on the semigroup $\Cal T$
of commuting endomorphisms of $G$, we can consider a barycenter of
the form $Pf(x) = \sum_{T \in \Cal T} \, \nu(T) f(Tx)$. For a
barycenter with finite support, we have seen that the decay, when
$\varphi_f$ is continuous, is of order $n^{{d-1 \over 2}}$. A
question is to estimate the decay when $\nu$ has an infinite support
and to study the asymptotic distribution of the normalized iterates.

For instance, if we $Pf(x) = \sum_{q \in \Cal P} \nu(q) f(qx)$,
where $\Cal P$ is the set of prime numbers, $(\nu(q), q \in \Cal P)$
a probability vector with $\nu(q) > 0$ for every prime $q$ and $f$
H\"olderian on the circle, what is the decay to 0 of $\|P^nf\|_2$ ?

A partial result is that, if $\nu$ has an infinite support, the
decay is faster than $Cn^{-r}$, for every $r \geq 1$. Indeed, this
can be deduced easily from the following observation:

Let $P_1$ and $P_2$ be two commuting contractions of $L^2(G)$, such
that $\|P_1^nf\|_2 \leq Mn^{-r}$ and let $\alpha \in ]0, 1], \beta =
1 - \alpha$. Then we have:
$$\|(\alpha P_1 + \beta P_2)^n f) \|_2 \leq  \sum_{k= 0}^n {n\choose k}
\alpha^k \beta^{n-k} \|P_1^k f\|_2.$$

Using the inequality of large deviation for the binomial law, we
obtain, with $c < 1$:
$$\sum_{k \leq {n \over 2\alpha}} {n\choose k} \alpha^k \beta^{n-k}
\leq c^n,$$ and therefore:
$$\sum_{k= 0}^n {n\choose k} \alpha^k \beta^{n-k} \|P_1^k f\|_2
\leq  M ({n \over 2\alpha})^{-r} - c^n \leq M' n^{-r}.$$

\vskip 4mm \goodbreak \subsection{\bf The torus case}
\label{endoTori}

\begin{nota} \label{notaGl} {\rm Let ${\Cal M}^*(\rho, \Z)$ denote
the semigroup of $\rho \times \rho$ non singular matrices with
coefficients in $\Z$ and  $GL(\rho, \Z)$ the group of matrices with
coefficients in $\Z$ and determinant $\pm 1$.

Every $A$ in ${\Cal M}^*(\rho, \Z)$ defines a surjective
endomorphism of $\T^\rho$, hence a measure preserving transformation
on $(\T^\rho, \mu)$ and a dual endomorphism on the group of
characters of $\T^\rho$ identified with $\Z^\rho$ (action by the
transposed of $A$). If $A$ is in $GL(\rho, \Z)$, it defines an
automorphism of $\T^\rho$. For simplicity, since the matrices are
commuting, we use the same notation for the matrix $A$, its action
on the torus and the dual endomorphism, without writing
transposition.}\end{nota}

{\it So, when $G$ is a torus $\T^\rho$, $\rho \geq 1$, we consider a
finite set of endomorphisms given by matrices $A_i$ in ${\Cal
M}^*(\rho, \Z)$. The group generated by the matrices $A_i$ in
$GL(\rho, \Q)$ is supposed to be torsion-free.}

The construction of Lemma \ref{embedd} can be describe in the
following way.  Let $p_i$ be the determinant of $A_i$, for each $i$.
If $\tilde G$ is the compact group dual of the discrete group
$\tilde \Z^\rho :=\{\k \,\Pi p_i^{\ell_i}, \k \in \Z^\rho, \ell_i
\in \Z \}$, then $\Z^\rho$ is a subgroup of $\tilde \Z^\rho$ and
$\tilde G$ has $\T^\rho$ as a factor.

It is well known that ergodicity for the action of a single $A \in
{\Cal M}^*(\rho, \Z)$ on $(\T^\rho, \mu)$ is equivalent to the
absence of eigenvalue root of 1 for $A$. Recall also Kronecker's
result: an integer matrix with all eigenvalues on the unit circle
has all eigenvalues roots of unity.

For a torus, total ergodicity is equivalent to the property that
$A^\n$ has no eigenvalue root of unity, for $\n \not = \underline
0$. Replacing $\n$ by a multiple, there is no $\n \not = \underline
0$ such that $A^\n$ has a fixed vector $v \not =\{0\}$. In other
words, total ergodicity is equivalent to say that the $\Z^d$-action
$\n =(n_1, ..., n_d) : (v \to A^\n v)$ on $\Z^\rho \ \stm0$ is free.

Using a common triangular representation over $\C$ for the commuting
matrices $A_j$, one sees that if $\lambda_{1j}, ..., \lambda_{\rho
j}$ are the eigenvalues of $A_j$ (with multiplicity), for $j = 1,
..., d$, this is equivalent to $\ (\prod_{j=1}^d \lambda_{ij}^{n_j}
= 1 \ \Rightarrow (n_1, ..., n_d) = \0)$, $\forall i \in \{1, ...,
\rho\}$.

\begin{lem} \label{irreduc1} Let $B \in {\Cal M}^*(\rho, \Z)$ be
a matrix with irreducible (over $\Q$) characteristic polynomial $P$.
Let $\{A_1, ..., A_d\}$ be $d$ matrices in ${\Cal M}^*(\rho, \Z)$
commuting with $B$. They generates a commutative semigroup of
endomorphisms on $\T^\rho$ which is totally ergodic, if and only if
for any $\n \in \Z^d \stm0$, $A^\n \not = Id$.
\end{lem}
\begin{proof} \ Since $P$ is irreducible, the eigenvalues of $B$ are
distinct. It follows that (on $\C$) the matrices $A_i$ are
simultaneously diagonalizable, hence are pairwise commuting. Now
suppose that there are $\n \in \Z^d \stm0$ and $v \in \Z^\rho
\setminus \{\0\}$ such that $A^\n v  = v$. Let $W$ be the subspace
of $\R^\rho$ generated by $v$ and its images by $B$. The restriction
of $A^\n$ to $W$ is the identity. $W$ is $B$-invariant, the
characteristic polynomial of the restriction of $B$ to $W$ has
rational coefficients and factorizes $P$. By the assumption of
irreducibility over $\Q$, this implies $W = \R^\rho$. Therefore
$A^\n$ is the identity.
\end{proof}

\begin{lem} \label{endoirred1} Let $\Cal S$ be the semigroup generated by
$d$ matrices $\{A_1, ..., A_d\}$ in ${\Cal M}^*(\rho, \Z)$ with the
irreducibility property like in Lemma \ref{irreduc1}, with
determinant $q_i$. If the numbers $\log |q_i|$ are linearly
independent over $\Q$, then $\Cal S$ is totally ergodic.
\end{lem}

\vskip 2mm {\bf Rate of decorrelation for automorphisms of the
torus}

The analysis for a general group $G$ relies on the absolute
convergence of the Fourier series of a function $f$ on $G$, which is
ensured, for the torus, if $f$ satisfies the following regularity
condition:
\begin{eqnarray}
|c_f(\k)| = O(\|k\|^{-\beta}), \text{ with } \beta > \rho.
\label{regFour2b2}
\end{eqnarray}
Nevertheless, as we will see now, for the torus, a weaker regularity
condition on $f$ can be used in the study of the spectral density
and the CLT. This is closely related to the rate of decorrelation
for regular functions, which is based on the following lemma:

\begin{lem} (D. Damjanovi\'c and A. Katok, \cite{DaKa10} for
automorphisms, M. Levine (\cite{Lev13} for endomorphisms)
\label{DaKa} If $(A^\n, \n \in \Z^d)$ is a totally ergodic
$\Z^d$-action on $\T^\rho$ by automorphisms, there are $\tau > 0$
and $C > 0$, such that for all $(\n, \k) \in \Z^d \times (\Z^\rho
\stm0)$ for which $A^\n \k\ \in \Z^\rho$.
\begin{eqnarray}
\|A^\n \k\| \geq C e^{\tau \|\n\|} \|\k\|^{-\rho}. \label{DaKaMin}
\end{eqnarray}
\end{lem}

The proof of the previous result for automorphisms uses the fact
that if $B$ is in ${\Cal M}^*(\rho, \Z)$ and $V$ a $m$-dimensional
eigenspace of $B$ such that $V \cap \Z^\rho = \{0\}$, then there
exists a constant $C$ such that, for every $\j \in \Z^\rho \,
\stm0$, the distance $d(\j,V)$ of $\j$ to $V$ satisfies $d(\j, V)
\geq C \|\j \|^{-m}$ (cf. \cite{Leo60c}, Katznelson \cite[Lemma
3]{Ka71}) and a result of \cite{BlMo71}. The extension to
endomorphisms was obtained by M. Levine in the recent paper
\cite{Lev13} mentioned in the introduction.

\vskip 3mm {\bf Regularity and Fourier series}

We need some results from the theory of approximation of functions
by trigonometric polynomials.

For $f\in L^2(\mathbb T^d)$, the rectangular Fourier partial sums of
$f$ are denoted by $S_{N_1,\ldots, N_d}(f)$. The {\it integral
modulus of continuity} of $f$ is defined as
$$\omega_2(\delta_1,\cdots,\delta_d, f)=\sup_{|\tau_1|\le
\delta_1,\ldots,|\tau_d|\le \delta_d}
\|f(x_1+\tau_1,\cdots,x_d+\tau_d)- f(x_1,\cdots,x_d)\|_2.$$

Let $J_{N_1,\ldots,N_d}(t_1,\cdots,t_d)=K_{N_1,\ldots,
N_d}^2(t_1,\cdots, t_d) /\|K_{N_1,\ldots,N_d}\|_{L^2(\mathbb
T^d)}^2$ be the $d$-dimensional Jackson's kernel, where
$K_{N_1,\ldots,N_d}$ is the $d$-dimensional Fej\'er kernel.

Clearly, $J_{N_1,\ldots,N_d}(t_1,\cdots,t_d)=J_{N_1}(t_1)\cdots
J_{N_d}(t_d)$. It is known that the 1-dimensional Jackson's kernel
satisfies the following moment relations:
\begin{equation}
\label{jackson-moments} \int_0^{\frac12}t^k\, J_{N}(t) \, dt =
O(N^{-k}), \ \forall N \geq 1, \ k=0,1,2.
\end{equation}

\begin{lem} \label{regl1} There exists a positive constant $C_d$
such that, for every $f \in L^2(\T^d)$, for every $N_1,\ldots,
N_d\ge1$, $\|J_{N_1,\ldots,N_d}* f-f\|_2\le
C_d\,\omega_2(\frac1{N_1}, \cdots,\frac1{N_d}, f)$.
\end{lem}
\begin{proof} Since $\omega_2(\delta_1,\cdots,\delta_d, f)$ is
increasing and subadditive with respect to $\delta_i$, we have for
any positive numbers $\lambda_i$:
$\omega_2(\lambda_1\delta_1,\cdots,\lambda_d\delta_d, f)
\le(\lambda_1+1) \cdots (\lambda_d+1)\,
\omega_2(\delta_1,\cdots,\delta_d, f)$.

Using this inequality and \eqref{jackson-moments}, we obtain:
\begin{eqnarray*}
&\|J_{N_1,\ldots,N_d}* f-f\|_2\le \int_{[-\frac12,\frac12[^d}
J_{N_1,\ldots,N_d}(\tau_1,\cdots,\tau_d) \, \|f(.-\tau_1,\cdots, .
-\tau_d)
-f\|_{L^2} \, d\tau_1\cdots d\tau_d\\
&\le 2^d\int_{[0,\frac12[^d} J_{N_1,\ldots,N_d}
(\tau_1,\cdots,\tau_d)\, \omega_2(\tau_1,\cdots,\tau_d,f) \, d\tau_1\cdots d\tau_d\\
&=2^d \int_{[0,\frac12[^d} J_{N_1,\ldots,N_d}(\tau_1,\cdots,\tau_d)
\ \omega_2(\frac{N_1\tau_1}{N_1},\cdots, \frac{N_d\tau_d} {N_d},f)
\, d\tau_1\cdots d\tau_d\\
&\le 2^d\omega_2(\frac{1}{N_1},\cdots,\frac{1}{N_d},f)
\int_{[0,\frac12[^d} (N_1\tau_1+1)\cdots(N_d\tau_d+1) \,
J_{N_1,\ldots,N_d}(\tau_1,\cdots,\tau_d) \, d\tau_1\cdots d\tau_d\\
& =2^d\omega_2(\frac{1}{N_1},\cdots,\frac{1}{N_d},f)\, \prod_{i=1}^d
\int_0^{\frac12}(N_i\tau_i+1)J_{N_i}(\tau_i) \, d\tau_i \le C_d \,
\omega_2(\frac{1}{N_1},\cdots,\frac{1}{N_d},f).
\end{eqnarray*}
\end{proof}

\begin{prop} \label{majApprox} There exists a positive constant $C_{d}$,
such that, for every $f\in L^2(\mathbb T^d)$ and
$N_1,\ldots,N_d\ge1$, we have $\|f-S_{N_1,\ldots, N_d}(f)\|_2 \le
C_{d}\, \omega_2(\frac1{N_1},\cdots, \frac1{N_d}, f)$.
\end{prop} \begin{proof}
For every trigonometric polynomial $P$ in $d$ variables of degree at
most $N_1\times\cdots\times N_d$, we have: $\|f-S_{N_1,\ldots,
N_d}(f)\|_2\leq \|f-P\|_2$. The result follows then from Lemma
\ref{regl1}.
\end{proof}

By Proposition \ref{majApprox}, the following condition on the
modulus of continuity:
\begin{eqnarray}
&&\text{ There are } \alpha > 1 \text{ and } C(f)<+\infty \text{
such that } \omega_{2}(\delta, ..., \delta, f) \leq C(f) \, (\ln {1
\over \delta})^{-\alpha}, \forall \delta > 0. \label{5.11}
\end{eqnarray}
implies:
\begin{eqnarray}
&&\|f - s_{N,...,N}(f)\|_2 \leq R(f) \,  (\ln N)^{-\alpha}, \text{
with } \alpha > 1. \label{ineq5.23b2}
\end{eqnarray}
One easily checks that (\ref{regFour2b2}) implies
(\ref{ineq5.23b2}).

In what follows in this subsection, $\n \to A^\n$ is a totally
ergodic $\Z^d$-action by {\it endomorphisms} on $\T^\rho$. Recall
that $\tilde f$ denotes the extension to $\tilde G$ of a function
$f$ on $G$ (here $G = \T^\rho$) and that we use the convention (*)
(i.e., we put $c_{A^\n \k}(f) =0$ if $A^\n\k \not \in \Z^\rho$). We
denote simply by $| . |$ the norm of an integral vector. Recall that
we do not write the transposition for the dual action of $A^{\n}$.
The proof of the following proposition is like that of the analogous
result in \cite{Leo60c}.

\begin{prop} \label{ThmRatedeCor} Let $f \in L_0^2(\T^\rho)$ satisfying
(\ref{ineq5.23b2}) and $f_1(x) := \sum_{\n \in \Cal{N}_1} c_\n (f)
e^{2\pi i\langle \n , x \rangle}$, where $\Cal{N}_1$ is a subset of
$\Z^\rho$. Then there is a finite constant $B(f)$ depending only on
$R(f)$ such that
\begin{eqnarray}
|\langle A^{\n} f_1, f_1\rangle| \leq  B(f) \|f_1\|_2
\|\n\|^{-\alpha}, \ \forall \n \not = \0. \label{ineq5.24}
\end{eqnarray}
\end{prop}
\begin{proof} It suffices to prove the result for $f$, since,
by setting $c_f(n)=0$ outside $N_1$, we obtain (\ref{ineq5.24}) with
the same constant $B(f)$ as shown by the proof. Let $\lambda, b, d$
such that $1 < \lambda < e^{\tau}, \ 1 < b < \lambda^{{1\over
\rho}}, \ \lambda b^{-\rho} = d
> 1$. We have for $\n \in \Z^d$:
\begin{eqnarray}
&&\langle A^{\n} f, f\rangle = \sum_{\k \, \in \Z^\rho} c_\k(f) \,
\overline c_{A^\n\k}(f) = \sum_{|\k| < b^{|\n|}} \ + \ \sum_{|\k|
\geq b^{|\n|}}. \label{ineq5.25}
\end{eqnarray}

From  Inequality (\ref{DaKaMin}) of Lemma \ref{DaKa}, we deduce
that, if $|\k| < b^{|\n|}$, then $|A^\n\k| \geq D \lambda^{\n} \,
|\k|^{-\rho} \geq D \lambda^{|\n|} \, b^{-\rho |\n|} = D d^{|\n|}, \
\n \not = \0$. It follows, for the first sum:
\begin{eqnarray*}
|\sum_{|\k| < b^{|\n|}} c_\k(f) \ \overline c_{A^\n\k}(f)| \leq
(\sum_{|\k| < b^{|\n|}} |c_\k(f)|^2)^\frac12 \ (\sum_{|\k| <
b^{|\n|}} |c_{A^\n\k}(f)|^2)^\frac12 \leq  \|f\|_2 \,\sum_{|\m| >
Dd^{|\n|}} |c_{\m}(f)|^2.
\end{eqnarray*}

By Parseval inequality and (\ref{ineq5.23b2}), there is a finite
constant $B_1(f)$ such that, for ${|\n|} \not = 0$:
\begin{eqnarray}
(\sum_{|\m| > Dd^{|\n|}} |c_{\m}(f)|^2)^\frac12 &&\leq \|f-
s_{[Dd^{|\n|}], ..., [Dd^{|\n|}]}\|_2 \leq {R(f) \over (\ln
[Dd^{|\n|}])^\alpha} \leq B_1(f) |\n|^{-\alpha}. \label{ineq5.28}
\end{eqnarray}

From the previous inequalities, it follows:
\begin{eqnarray}
|\sum_{|\k| < b^{|\n|}} c_n(f) \, \overline c_{A^\n\k}(f)| \leq
B_1(f) \|f\|_2 {|\n|}^{-\alpha}, \forall {|\n|} \not = 0.
\label{ineq5.29}
\end{eqnarray}

Analogously, for the second sum in (\ref{ineq5.25}) we obtain
\begin{eqnarray} |\sum_{|\k| \geq b^{|\n|}} c_\k(f) \,
\overline c_{A^\n\k}(f)| \leq  B_2(f) \|f\|_2 {|\n|}^{-\alpha}, \,
|\n| \not = 0. \label{ineq5.30}
\end{eqnarray}

Taking $B(f)= B_1(f) + B_2(f)$, (\ref{ineq5.24}) follows from
(\ref{ineq5.25}), (\ref{ineq5.29}), (\ref{ineq5.30}).
\end{proof}

\vskip 3mm The following theorem has the same conclusion as Theorem
\ref{condAbsConv}, but requires a weaker regularity condition.
\begin{thm} \label{ContabsConv} If $f$ satisfies (\ref{ineq5.23b2}),
(in particular if $f$ satisfies the regularity condition
(\ref{5.11})), then $\sum_{\n \in \Z^d} |\langle A^{\n}f, f\rangle|
< \infty$, the variance $\sigma^2(f)$ exists, $\sigma^2(f) =
\sum_{\n \in \Z^d} \langle A^{\n}f, f\rangle$, the density
$\varphi_f$ of the spectral measure of $f$ is continuous.

Moreover, there is a constant $C$ such that, if $\Cal{N}$ is any
subset of $\Z^\rho$ and $f_1(x) = \sum_{\k \in \Cal{N}} c_\k(f)
e^{2\pi i\langle \k , x \rangle}$, then $\sigma(f - f_1) \leq C
\|f-f_1\|_2$.
\end{thm}
\begin{proof} The Fourier coefficients of $\varphi_f$ are $\langle A^\n f, f
\rangle$. The previous proposition implies [(\ref{ineq5.23b2}) $\
\Rightarrow \ \sum_{\n \in \Z^d} |\langle A^\n f, f\rangle| <
+\infty$] and the second statement.
\end{proof}

\vskip 3mm {\bf Coboundary characterization}

Using the previous result on the decay of correlation, let us give a
sufficient condition on the Fourier series of $f$ for the coboundary
characterization. Recall that $J$ denotes a section of the
$\Z^d$-action by automorphisms on $\Z^\rho$ (cf. Remark
\ref{sectJ}).
\begin{rem} \label{sectJ}
 For $G = \T^\rho$, each character is identified to an
element $\k$ of $\Z^\rho$. It is useful to choose the section in the
following way. For a fixed $\el$, the set $\{A^\k \el, \k \in
\Z^d\}$ is discrete and $\lim_{\|\k\| \to \infty} \|A^\k \el\| =
+\infty$. Therefore the minimum of the norm is achieved for some
value of $\k$. We can choose an element $\j$ in each class modulo
the action of $\tilde {\Cal S}$ on $\Z^\rho$, which achieves the
minimum of the norm. By this choice, we have
\begin{eqnarray}
\|\j\| \leq \|A^\k \j\|, \, \forall \j \in J, \, \k \in \Z^d.
\label{majJK0}
\end{eqnarray}
\end{rem}
The sufficient condition (\ref{condConv2}) given in Lemma
\ref{cobCond24} for the coboundary representation reads in the
algebraic framework
\begin{eqnarray}
\sum_{j \in J} \sum_{\k \in \Z^d} (1 +\|\k\|^d) \, |c_f(A^\k j)| <
\infty. \label{condConvAuto}
\end{eqnarray}
\begin{thm} \label{coboundChar} If $|c_f(\k)| = O(\|k\|^{-\beta}),
\text{ with } \beta > \rho$, we have $\sigma^2(f) =0$ if and only if
$f$ is a mixed coboundary: there are continuous functions $u_i$, $i=
1, ..., d$ such that
\begin{eqnarray}
f= \sum_{i=1}^d (I- A_i) u_i. \label{cobound12}
\end{eqnarray}
\end{thm}
\begin{proof} Let $\varepsilon \in \, ]0, \beta - \rho[$
and $\delta := (\beta - \rho - \varepsilon) / (\beta (1 + \rho))$,
we have $\delta \beta\rho -\beta(1-\delta)  = - (\rho +
\varepsilon)$.

There is a constant $C_1$ such that $\|\k\|^d \, e^{-\delta \beta
\tau \|\k\|} \leq C_1, \, \forall \k \in \Z^d$.

According to (\ref{DaKaMin}), we have $|c_f(A^\k \j)| \leq C \|A^\k
\j\|^\beta  \leq C e^{-\beta \tau \|\k\|} \, \|\j\|^{\beta \rho}$;
hence
\begin{eqnarray}
e^{\delta \beta \tau \|\k\|} \, |c_f(A^\k j)|^\delta \leq C
\|\j\|^{\delta \beta\rho}. \label{majkj1}
\end{eqnarray}
Recall that, for every $\el \in \Z^\rho \stm0$, there is a unique
pair $(k, j) \in \Z^d \times J$ such that $A^\k \j = \el$. Therefore
we have, using Inequality (\ref{majJK0}) (see Remark \ref{sectJ}):
\begin{eqnarray*}
&&\sum_{\j \in J}\sum_{\k \in \Z^d} \|\k\|^d \, |c_f(A^\k \j)| =
\sum_{\j \in J} \sum_{\k \in \Z^d} \|\k\|^d \, |c_f(A^\k \j)|^\delta
|c_f(A^\k \j)|^{1-\delta} \\
&& \ \ \leq C_1 \sum_{\j \in J}\sum_{\k \in \Z^d}  e^{\delta \beta
\tau \|\k\|} \, |c_f(A^\k \j)|^\delta \, |c_f(A^\k \j)|^{1-\delta} \
\leq C_2 \sum_{\j \in J}\sum_{\k \in \Z^d} \|\j\|^{\delta
\beta\rho} \, |c_f(A^\k \j)|^{1-\delta} \\
&&  \  \leq C_2 \sum_{\j \in J}\sum_{\k \in \Z^d} \|A^\k
\j\|^{\delta \beta\rho} \, |c_f(A^\k \j)|^{1-\delta} \text{ by
(\ref{majkj1})
and (\ref{majJK0})}\\
&& \ \ = C_2 \sum_{\el \in \Z^\rho \stm0} \|\el\|^{\delta \beta\rho}
\, |c_f(\el)|^{1-\delta} \leq C_3 \sum_{\el \in \Z^\rho \stm0}
\|\el\|^{\delta \beta\rho -\beta(1-\delta)} \leq C_3 \sum_{\el \in
\Z^\rho \stm0} \|\el\|^{-(\rho + \varepsilon)}< +\infty.
\end{eqnarray*}
This implies that (\ref{condConvAuto}) is satisfied. Since here the
functions involved in the proof of Lemmas \ref{cobRepr} and
\ref{cobCond24} are characters, hence continuous and uniformly
bounded, we have continuity of the functions $u_i$ in the
representation (\ref{cobound10A}). \end{proof}

Now we consider the CLT when $G$ is the torus $\T^d$. As mentioned
in the introduction, an analogous result for $d$-dimensional
rectangles and a class of regular functions was recently obtained by
M. Levine (\cite{Lev13}).

\begin{thm} \label{tclRegFoln} Let $\n \to A^\n$ be a totally ergodic
$\N^d$-action by commuting matrices on $\T^\rho$. Let $(R_n)_{n \geq
1}$ be a F\o{}lner summation sequence.

1) If $f$ satisfies (\ref{ineq5.23b2}), in particular if $f$
satisfies the regularity condition (\ref{5.11}), we have
$\sigma^2(f) = \varphi_f(0)$ and
$$ (\sum_{\el \in \N^d} R_n(\el)^2)^{-\frac12} \, \sum_{\el \in \N^d} R_n(\el) \, f(A^\el \, .)
 \overset{distr} {\underset{n \to \infty} \longrightarrow } \Cal N(0,\sigma^2(f)).$$

2) If $f$ satisfies (\ref{regFour2b2}) (i.e., $|c_f(\k)| =
O(\|k\|)^{-\beta}, \text{ with } \beta > \rho$), then $\sigma^2(f) =
0$ if and only if\footnote{This gives a test of non degeneracy of
the limiting law in terms of periodic points.} there are continuous
functions $u_t$ on $\T^\rho$,
 for $t= 1, ..., d$, such that $f= \sum_{t=1}^d (I- A_t) u_t$.
\end{thm}
\begin{proof} 1) The proof of the first statement is like
the proof of Theorem \ref{tclPol0A}. Here we use the inequality
$\sigma(f - f_s) \leq C \|f-f_s\|_2$, where the constant $C$ does
not depend on $s$, given by Theorem \ref{ContabsConv}.

2) The second assertion follows from Theorem  \ref{coboundChar}.
\end{proof}

\vskip 5mm \subsection {\bf Appendix: examples of $\Z^d$-actions by
automorphisms}

The aim of this appendix is to give explicit examples of commuting
matrices generating totally ergodic $\Z^d$-actions on tori. We
mainly recall some known facts. (See in particular \cite{KKS02} and
\cite{DaKa10} for the construction of $\Z^d$-actions by
automorphisms on the torus.)

The construction of  $\Z^d$-action by automorphisms on $\T^\rho$ is
linked to the groups of units in number fields. Following
\cite{KKS02}, let us recall some facts.

Let $M \in GL(\rho, \Z)$ be a matrix with an irreducible
characteristic polynomial $P=P(M)$ and hence distinct eigenvalues.
The centralizer of $M$ in $\Cal M(n, \Q)$ can be identified with the
ring of all polynomials in $M$ with rational coefficients modulo the
principal ideal generated by the polynomial $P(M)$, and hence with
the field $K = \Q(\lambda)$, where $\lambda$ is an eigenvalue of
$M$, by the map $\gamma : p(A) \to p(\gamma)$ with $p \in \Q[x]$.

By Dirichlet's theorem, if $P$ has $d_1$ real roots and $d_2$ pairs
of complex conjugate roots, then there are $d_1 + d_2-1$ fundamental
units in the group of units in the ring of integers in $K(P)$. This
provides a totally ergodic $\Z^{d_1+d_2 -1}$-action by automorphisms
on $\T^\rho$.

Explicit computation of examples relies on an algorithm (see
\cite{Co93}). The first computed examples appeared with the
development of the computers. Even nowadays computations are limited
to low dimensional examples.

\vskip 3mm {\bf Examples for $\T^3$}

To give a concrete example for $\T^3$, we explicit a pair $A, B$ of
matrices in $SL(3, \Z)$ such that $\{A, B\}$ generates a free action
in $\Z^2$.

We start with a integer polynomial $P(X) = -X^3 +qX +n$ which is
irreducible over $\Q$ and its companion matrix:
\begin{eqnarray*} M = \left(
\begin{matrix} 0 & 1 & 0 \cr 0 & 0 & 1 \cr n & q & 0 \cr
\end{matrix}
\right).
\end{eqnarray*}

Let $K(P)$ denote the number field associated to $P$. Suppose that
$K(P)$ belongs to any of the tables where the characteristics of the
first cubic real fields $K(P)$ are listed. Let $\theta$ be a root of
$P$. The table gives a pair of fundamental units for the group of
units in the ring of integers in $K(P)$ of the form $P_1(\theta)$,
$P_2(\theta)$, where $P_1$, $P_2$ are two integer polynomials. Then
the matrices $A_1= P_1(M)$ and $A_2= P_2(M)$ give a system of
generators of the group of matrices in $GL(3, \Z)$ commuting with
$M$, generating a totally ergodic $\Z^2$-action on $\T^3$ by
automorphisms.

\vskip 3mm {\it Explicit examples}

1) (from the table in \cite{SteRud76}) Let us consider the
polynomial $P(X) = X^3 -12X -10$ and its companion matrix
\begin{eqnarray*} M = \left(
\begin{matrix} 0 & 1 & 0 \cr 0 & 0 & 1 \cr 10 & 12 & 0 \cr
\end{matrix}
\right).
\end{eqnarray*}

Let $\theta$ be a root of $P$. The table gives the pair of
fundamental units for the algebraic group associated to $P$:
$$P_1(\theta) = \theta^2 -3\theta -3, \ P_2(\theta) = -\theta^2
+\theta +11.$$

The matrices $A_1= P_1(M)$ and $A_2= P_2(M)$ give a system of
generators of the group of matrices in $GL(3, \Z)$ commuting with
$M$. They generate a totally ergodic action of $\Z^2$ by
automorphisms on $\T^3$. They have 3 real eigenvalues and $\det(A_1)
= 1, \det(A_2) = -1$.
\begin{eqnarray*} A_1 = \left(
\begin{matrix} -3 & -3 & 1 \cr 10 & 9 & -3 \cr -30 & -26 & 9 \cr
\end{matrix}
\right), \ \ A_2 = \left(
\begin{matrix} 11 & 1 & -1 \cr -10 & -1 & 1 \cr 10 & 2 & -1 \cr
\end{matrix}
\right).
\end{eqnarray*}

\vskip 3mm 2) (from tables in \cite{SteRud76} and in \cite{Co93})
Let us consider now the polynomial $P(X) = X^3 -9X -2$ and its
companion matrix
\begin{eqnarray*} M = \left(
\begin{matrix} 0 & 1 & 0 \cr 0 & 0 & 1 \cr 2 & 9 & 0 \cr
\end{matrix}
\right).
\end{eqnarray*}

Let $\theta$ be a root of $P$. The table gives the pair of
fundamental units for the algebraic group associated to $P$:
$$P_1(\theta) = 3\theta^2 -9\theta -1, \ P_2(\theta) = 2 \theta^2
- 4 \theta - 1.$$

The matrices $A_1= P_1(M)$ and $A_2= P_2(M)$ give a system of
generators of the group of matrices in $GL(3, \Z)$ commuting with
$M$. They generate a totally ergodic action of $\Z^2$ by
automorphisms on $\T^3$. They have 3 real eigenvalues and $\det(A_1)
= 1, \det(A_2) = -1$.

\begin{eqnarray*} A_1 = \left(
\begin{matrix} -3 & -3 & 1 \cr 10 & 9 & -3 \cr -30 & -26 & 9 \cr
\end{matrix}
\right), \ \ A_2 = \left(
\begin{matrix} 11 & 1 & -1 \cr -10 & -1 & 1 \cr 10 & 2 & -1 \cr
\end{matrix}
\right).
\end{eqnarray*}

Remark that in \cite{SteRud76} a different set of generators is
given. The polynomials are
$$P_1'(\theta) = 85\theta^2 - 245\theta -59, \ P_2'(\theta) = - 18 \theta^2
+ 4 \theta + 161.$$

The matrices $A_1'= P_1'(M)$ and $A_2'= P_2'(M)$ give another pair
of generators of the group of matrices in $GL(3, \Z)$ commuting with
$M$. The relations between the two pairs are:
$$ A_1' = A_1 A_2, \ A_2' = A_1^{-1}.$$

\vskip 3mm {\bf A simple example on $\T^4$}

If $P(X) = X^4 + a X^3 + b X^2 + a X + 1$, the polynomial $P$ has
two real roots: $\lambda_0, \lambda_0^{-1}$ and two complex
conjugate roots of modulus 1: $\lambda_1, \overline \lambda_1$. Let
$\sigma_j= \lambda_j  + \overline \lambda_j$, $j=0,1$. They are
roots of $Z^2 -a Z +b-2 =0$.

Under the conditions: $a^2 -4b+8 > 0$,  $a > 4$, $b > 2$, $2a >
b+2$, i.e., (since $2a-2 \leq \frac14a^2 +2$) $$2 < b < \frac14a^2
+2, \ a > 4,$$ $\lambda_0, \lambda_0^{-1}$ are solution of
$\lambda^2 - \sigma_0 \lambda +1 = 0$, and $\lambda_1, \overline
\lambda_1$ are solutions of $\lambda^2 - \sigma_1 \lambda +1 = 0$,
where
$$\sigma_0 = -\frac12 a - \frac12 \sqrt{a^2 -4b +8},
\ \sigma_1 = -\frac12 a + \frac12 \sqrt{a^2 -4b +8}.$$

The polynomial $P$ is not factorizable over $\Q$. Indeed, suppose
that $P = P_1 P_2$ with $P_1$, $P_2$ with rational coefficients and
degree $\geq 1$. Since the roots of $P$ are irrational, the degrees
of $P_1$ and $P_2$ are 2. Necessarily their roots are, say,
$\lambda_1, \overline \lambda_1$ for $P_1$, $\lambda_0,
\lambda_0^{-1}$ for $P_2$. The sum $\lambda_1 + \overline
\lambda_1$, root of $Z^2 -a Z +b-2 =0$, is not rational and the
coefficients of $P_1$ are not rational. Let
\begin{eqnarray*} A :=
\left( \begin{matrix} 0 & 1 & 0 & 0 \cr 0 & 0 & 1 & 0 \cr 0 & 0 & 0
& 1 \cr -1 & -a & -b & -a \cr
\end{matrix}
\right), B = A+I.
\end{eqnarray*}
From the irreducibility over $\Q$, it follows that, if there is a
non zero fixed integral vector for $A^k B^{\ell}$, where $k, \ell$
are in $\Z$, then we have $A^k B^{\ell} = Id$. This implies:
$\lambda_1^k \, (\lambda_1 -1)^k = 1$, hence, since we have
$|\lambda_1| = 1$, it follows $|\lambda_1 -1| = 1$ which clearly is
not true.

\vskip 3mm {\it Example}: $P(X) = X^4 + 5X^3 +7 X^2 + 5X +1$. If $A$
is the companion matrix, then $A$ and $A+1$, with characteristic
polynomials $P(X)$ and $X^4 + X^3 - 2 X^2 + 2 X - 1$ respectively,
generate a $\Z^2$-totally ergodic action on $\T^4$.

\begin{eqnarray*}
A = \left( \begin{matrix} 0 & 1 & 0 & 0 \cr 0 & 0 & 1 & 0 \cr 0 & 0
& 0 & 1 \cr -1 & -5 & -7 & -5 \cr
\end{matrix}
\right), \ B= A+I =\left( \begin{matrix}  1 & 1 & 0 & 0 \cr 0 & 1& 1
& 0 \cr 0 & 0 & 1 & 1 \cr -1 & -5 & -7 & -4 \cr
\end{matrix}
\right).
\end{eqnarray*}

This elementary example gives only a $\Z^2$-action on $\T^4$. A
question is to produce an example with full dimension 3.

\vskip 3mm 3) {\bf Construction by blocks} \ Let $M_1, M_2$ be two
ergodic matrices respectively of dimension $d_1$ and $d_2$. Let
$p_i, q_i$, $i=1,2$ be two pairs of integers such that $p_1q_2 - p_2
q_1 \not = 0$. On the torus $\T^{d_1+d_2}$ we obtain a
$\Z^2$-totally ergodic action by taking $A_1, A_2$ of the following
form:
\begin{eqnarray*}
A_1 = \left( \begin{matrix} M_1^{p_1} & 0 \\ 0 & M_2^{q_1} \cr
\end{matrix}
\right), \ A_2 = \left( \begin{matrix} M_1^{p_2} & 0 \\ 0 &
M_2^{q_2} \cr
\end{matrix}
\right).
\end{eqnarray*}

Indeed, if there exists $v = \left( \begin{matrix} v_1 \\ v_2\cr
\end{matrix} \right) \in \Z^{d_1 + d_2} \setminus \{0\}$ invariant by $A_1^n
A_2^\ell$, then $M_1^{np_1 + \ell p_2} v_1 = v_1, \ M_2^{nq_1 + \ell
q_2} v_2 = v_2$, which implies $np_1 + \ell p_2 = 0$, $nq_1 + \ell
q_2 = 0$; hence $n = \ell = 0$.

\vskip 3mm This is a method to obtain explicit free $\Z^2$-actions
on $\T^4$. The same method gives explicit free $\Z^3$-actions on
$\T^5$ (by using a $\Z$-action on $\T^2$ and a $\Z^2$-action on
$\T^3$).

We do not know explicit examples of full dimension, i.e., with 3
independent generators on $\T^4$, or with 4 independent generators
on $\T^5$.

\vskip 3mm
{\bf Acknowlegements} This research was carried out during visits of
the first author to the University of Rennes 1 and of the second
author to the Center for Advanced Studies in Mathematics at Ben
Gurion University. The authors are grateful to their hosts for their
support. They thank Y. Guivarc'h, S. Le Borgne and M. Lin for
helpful discussions.


\vskip 20mm
\begin{thebibliography}{123}

\bibitem[1]{BeGu11}
B. Bekka, Y. Guivarc'h: {\it On the spectral theory of groups of
affine transformations on compact nilmanifolds}, arXiv{1106.2623}.

\vskip 2mm
\bibitem[2] {Bill99}
 P. Billingsley: {\it Convergence of probability measures},
 Second edition. John Wiley \& Sons, New York, 1999.

\vskip 2mm
\bibitem [3] {BlMo71}
P. E. Blanksby, H. L. Montgomery: Algebraic integers near the unit
circle, Acta Arith. 18 (1971), 355-369.

\vskip 2mm
\bibitem [4] {CohCo12}
G. Cohen, J.-P. Conze: The CLT for rotated ergodic sums and related
processes, Disc. Cont. Dynam. Syst. 33, no 9 (2013).

\vskip 2mm
\bibitem [5] {Co93}
H. Cohen: A course in computational algebraic number theory.
Graduate Texts in Mathematics, 138. Springer-Verlag, Berlin, 1993.

\vskip 2mm
\bibitem [6] {DaKa10}
D. Damjanovi\'c, A. Katok: Local rigidity of partially hyperbolic
actions I. KAM method and $\Z^k$-actions on the torus, Annals of
Math, 2010.

\vskip 2mm
\bibitem[7]{Ev84} J.H. Evertse, On equations in
$S$-units and the Thue-Mahler equation, Invent. Math. {\bf 75}
(1984), no. 3, 561-584.

\vskip 2mm
\bibitem [8] {EvScSch02}
J.-H. Evertse, H. P. Schlickewei and W. M. Schmidt: Linear equations
in variables which lie in a multiplicative group, Annals of
Mathematics, 155 (2002), 807-836.

\vskip 2mm
\bibitem [9] {FreSho31}
M. Fr\'echet, J. Shohat: A proof of the generalized second limit
theorem in the theory of probability, Trans. Amer. Math. Soc. 33
(1931), no. 2, 533-543.

\vskip 2mm
\bibitem[10]{FuPe01}
K. Fukuyama, B. Petit: Le th\'eor\`eme limite central pour les
suites de R. C. Baker. (French) [Central limit theorem for the
sequences of R. C. Baker] Ergodic Theory Dynam. Systems 21 (2001),
no. 2, p. 479-492.

\vskip 2mm
\bibitem[11]{FuSh99}
A. Furman, Ye. Shalom: Sharp ergodic theorems for group actions and
strong ergodicity, Ergodic Theory Dynam. Systems 19 (1999), no. 4,
p. 1037-1061.

\vskip 2mm
\bibitem[12]{GneKol54}
B. V. Gnedenko, A. N. Kolmogorov: Limit distributions for sums of
independent random variables. Translated and annotated by K. L.
Chung. With an Appendix by J. L. Doob. Addison-Wesley Publishing
Company, Inc., Cambridge, Mass., 1954. ix+264 pp.

\vskip 2mm
\bibitem[13]{KKS02}
A. Katok, S. Katok, K. Schmidt: Rigidity of measurable structure for
$Z^d$-actions by automorphisms of a torus. Comment. Math. Helv. 77
(2002), no. 4, 718-745.

\vskip 2mm
\bibitem[14]{Ka71}
Y. Katznelson: Ergodic automorphisms of $\T^n$ are Bernoulli shifts.
Israel J. Math. 10 (1971), 186-195.

\vskip 2mm
\bibitem[15]{Le78}
F. Ledrappier: Un champ markovien peut \^etre d'entropie nulle et
m\'elangeant (French), C. R. Acad. Sci. Paris S\'er. A-B 287 (1978),
no. 7, A561-A563.

\vskip 2mm
\bibitem[16]{Leo60a}
V. P. Leonov: The use of the characteristic functional and
semi-invariants in the ergodic theory of stationary processes, Dokl.
Akad. Nauk SSSR 133 523-526 (Russian); translated as Soviet Math.
Dokl. 1 (1960) 878-881.

\vskip 2mm
\bibitem[17]{Leo60b}
V. P. Leonov: On the central limit theorem for ergodic endomorphisms
of compact commutative groups, (Russian) Dokl. Akad. Nauk SSSR 135
(1960) 258-261.

\vskip 2mm
\bibitem[18]{Leo60c}
V. P. Leonov: Some applications of higher semi-invariants to the
theory of stationary random processes, Izdat. ``Nauka'', Moscow
1964, 67 pp. (Russian).

\vskip 2mm
\bibitem[19]{Lev13}
M. Levine: Central limit theorem for $\Z_+^d$-actions by toral
endomorphisms, Electron. J. Probability, 18 (2013), no. 35, p.1-42.

\vskip 2mm
\bibitem[20]{Ph94}
W. Philipp: Empirical distribution functions and strong
approximation theorems for dependent random variables. A problem of
Baker in probabilistic number theory. Trans. Amer. Math. Soc. 345
(1994), no. 2, 705-727.

\vskip 2mm
\bibitem[21] {Roh61}
V. A. Rohlin: The entropy of an automorphism of a compact
commutative group. (Russian) Teor. Verojatnost. i Primenen. 6 (1961)
351-352.

\vskip 2mm
\bibitem[22] {Schl90}
H.P. Schlickewei: S-unit equations over number fields, Invent. Math.
102, 95-107 (1990).

\vskip 2mm
\bibitem[23] {Sch90}
K. Schmidt: Automorphisms of compact abelian groups and affine
varieties. Proc. London Math. Soc. (3) 61 (1990), no. 3, 480-496.

\vskip 2mm
\bibitem[24] {SchWar93}
K. Schmidt, T. Ward: Mixing automorphisms of compact groups and a
theorem of Schlickewei, Invent. Math. 111 (1993), no. 1, 69-76.

\vskip 2mm
\bibitem[25]{Sc95}
K. Schmidt: Dynamical systems of algebraic origin. Progress in
Mathematics, 128. Birkh\"auser Verlag, Basel, 1995.

\vskip 2mm
\bibitem[26]{SteRud76}
R. Steiner, R. Rudman: On an algorithm of Billevich for finding
units in algebraic fields, Math. Comp. 30 (1976), no. 135, 598-609.

\vskip 2mm
\bibitem[27]{Zyg}
A. Zygmund: {\it Trigonometric series}, Second edition, Cambridge
University Press, London-New York 1968.
\end{thebibliography}
\end{document}